\documentclass[12pt,reqno,draft]{article} 
\usepackage{amsmath,amssymb,amsthm,amsfonts, indentfirst}
\usepackage{enumerate,color,bm}
\usepackage{mathtools}
\makeatletter
\def\@cite#1#2{[{{\bfseries #1}\if@tempswa , #2\fi}]}
\renewcommand{\section}{%
\@startsection{section}{1}{\z@}
{0.5truecm plus -1ex minus -.2ex}%
{1.0ex plus .2ex}{\bfseries\large}}
\def\@seccntformat#1{\csname the#1\endcsname.\ }
\makeatother

\numberwithin{equation}{section} 
\theoremstyle{theorem}
\newtheorem{thm}{Theorem}[section]
\newtheorem{cor}[thm]{Corollary}
\newtheorem{lem}[thm]{Lemma}

\theoremstyle{definition}

\newtheorem{remark}{Remark}[section]

\newtheorem*{prth1.1}{Proof of Theorem 1.1}
\newtheorem*{prth2.2}{Proof of Theorem 2.2}
\newtheorem*{prcor1.2}{Proof of Corollary 1.2}
\newtheorem*{prth1.3}{Proof of Theorem 1.3}

\newcommand{\ep}{\varepsilon}
\newcommand{\pa}{\partial}
\newcommand{\Rn}{\mathbb{R}^n}
\newcommand{\R}{\mathbb{R}}

\newcommand{\cl}[1]{{\overline#1}}

\newcommand{\tmax}{T_{\rm max}}

\begin{document}
\footnote[0]{
    2010{\it Mathematics Subject Classification}\/. 
    Primary: 35A01; 
    Secondary: 35Q92, 92C17.
    }
\begin{center} 
    \Large{{\bf 
    Boundedness in a fully parabolic attraction-repulsion chemotaxis system 
    with nonlinear diffusion and signal-dependent sensitivity
    }}%
\end{center}
\vspace{5pt}
\begin{center}
    Yutaro Chiyo, 
    Tomomi Yokota%
      \footnote{Corresponding author.}%
      \footnote{Partially supported by Grant-in-Aid for
      Scientific Research (C), No.\,21K03278.}
    \footnote[0]{
    E-mail: 
    {\tt ycnewssz@gmail.com}, 
    {\tt yokota@rs.tus.ac.jp}
    }\\
    \vspace{12pt}
    Department of Mathematics, 
    Tokyo University of Science\\
    1-3, Kagurazaka, Shinjuku-ku, 
    Tokyo 162-8601, Japan\\
    \vspace{2pt}
\end{center}
\begin{center}    
    \today
\end{center}

\vspace{2pt}
\newenvironment{summary}
{\vspace{.5\baselineskip}\begin{list}{}{%
     \setlength{\baselineskip}{0.85\baselineskip}
     \setlength{\topsep}{0pt}
     \setlength{\leftmargin}{12mm}
     \setlength{\rightmargin}{12mm}
     \setlength{\listparindent}{0mm}
     \setlength{\itemindent}{\listparindent}
     \setlength{\parsep}{0pt}
     \item\relax}}{\end{list}\vspace{.5\baselineskip}}
\begin{summary}
{\footnotesize {\bf Abstract.}
    This paper deals with the quasilinear fully parabolic attraction-repulsion 
    chemotaxis system 
%
\begin{align*}
\begin{cases}
      u_t=\nabla \cdot (D(u)\nabla u)
            -\nabla \cdot (G(u)\chi(v)\nabla v)
            +\nabla\cdot(H(u)\xi(w)\nabla w),
            & x \in \Omega,\ t>0,
    \\[1.05mm]
      v_t=d_1\Delta v+\alpha u-\beta v,
            & x \in \Omega,\ t>0,
    \\[1.05mm]
      w_t=d_2\Delta w+\gamma u-\delta w,
            & x \in \Omega,\ t>0,
\end{cases}
\end{align*}
    under homogeneous Neumann boundary conditions and initial conditions, 
    where $\Omega \subset \Rn$ $(n \ge 1)$ is a bounded domain with
    smooth boundary, 
    $d_1, d_2, \alpha, \beta, \gamma, \delta>0$ are constants. 
    Also, the diffusivity $D$, the density-dependent sensitivities $G, H$ 
    fulfill $D(s)=a_0(s+1)^{m-1}$ with $a_0>0$ and $m \in \R$; 
    $0 \le G(s) \le b_0(s+1)^{q-1}$ with $b_0>0$ and $q<\min\{2,\  m+1\}$; 
    $0 \le H(s) \le c_0(s+1)^{r-1}$ with $c_0>0$ and $r<\min\{2,\  m+1\}$, 
    and the signal-dependent sensitivities $\chi, \xi$ satisfy 
    $0\le\chi(s)\le \frac{\chi_0}{s^{k_1}}$ with $\chi_0>0$ and $k_1>1$; 
    $0\le\xi(s)\le \frac{\xi_0}{s^{k_2}}$ with $\xi_0>0$ and $k_2>1$.
    Global existence and boundedness in the case that $w=0$
    were proved by Ding (J. Math.\ Anal.\ Appl.; 2018;461;1260--1270)
    and Jia--Yang (J. Math.\ Anal.\ Appl.; 2019;475;139--153). 
    However, there is no work on the above fully parabolic 
    attraction-repulsion chemotaxis system with nonlinear diffusion 
    and signal-dependent sensitivity. 
    This paper develops global existence and 
    boundedness of classical solutions to the above system 
    by introducing a new test function.
}
\end{summary}
\vspace{10pt}

\newpage

\section{Introduction} \label{Sec1}

In this paper we consider the fully parabolic attraction-repulsion 
chemotaxis system with nonlinear diffusion and signal-dependent sensitivity, 
%
\begin{align}\label{P}
\begin{cases}
      u_t=\nabla \cdot (D(u)\nabla u)
            -\nabla \cdot (G(u)\chi(v)\nabla v)
            +\nabla\cdot(H(u)\xi(w)\nabla w),
            & x \in \Omega,\ t>0,
    \\[1.05mm]
      v_t=d_1\Delta v+\alpha u-\beta v,
            & x \in \Omega,\ t>0,
    \\[1.05mm]
      w_t=d_2\Delta w+\gamma u-\delta w,
            & x \in \Omega,\ t>0,
    \\[1.5mm]
      \nabla u \cdot \nu=\nabla v \cdot \nu=\nabla w \cdot \nu=0,
            & x \in \pa\Omega,\ t>0,
    \\[1.05mm]
      u(x,0)=u_0(x),\ v(x,0)=v_0(x),\ w(x,0)=w_0(x),
            & x \in \Omega,
\end{cases}
\end{align}
%
where $\Omega \subset \R^n$ ($n \ge 1$) is a bounded domain 
with smooth boundary $\pa \Omega$; 
$\nu$ is the outward normal vector to $\pa \Omega$; 
$d_1, d_2, \alpha, \beta, \gamma, \delta$ are positive constants; 
$D, G, H, \chi, \xi \ge 0$ are known functions of which 
the typical examples are given by $D(u)=(u+1)^{m-1}$, $G(u)=(u+1)^{q-1}$, 
$H(u)=(u+1)^{r-1}$, $\chi(v)=\frac{1}{v^{k_1}}$, $\xi(w)=\frac{1}{w^{k_2}}$,
where $m,q,r \in \R$, $k_1, k_2>1$. 
The initial data $u_0, v_0, w_0$ are supposed to satisfy that 
%
\begin{align}
    &u_0 \in C^0(\cl{\Omega}),\quad u_0 \ge 0\ {\rm in}\ \cl{\Omega}
                                          \quad {\rm and}\quad u_0 \neq 0,
      \label{condiu0}\\
    &v_0 \in W^{1, \infty}(\Omega)\quad {\rm and}
                                               \quad v_0>0\ {\rm in}\ \cl{\Omega},
      \label{condiv0}\\
    &w_0 \in W^{1, \infty}(\Omega)\quad {\rm and}
                                               \quad w_0>0\ {\rm in}\ \cl{\Omega}.
      \label{condiw0}
\end{align}
%

%
In the system \eqref{P}, the function $u$ shows the cell density, 
and the functions $v$ and $w$ represent 
the concentrations of attractive and 
repulsive chemical substances, respectively.
The system \eqref{P} is a generalization of the original attraction-repulsion chemotaxis system 
\begin{align}\label{ARconst}
\begin{cases}
      u_t=\Delta u
            -\chi\nabla\cdot (u\nabla v)
            +\xi\nabla\cdot(u\nabla w),
    \\[1.05mm]
      v_t=d_1\Delta v+\alpha u-\beta v,
    \\[1.05mm]
      w_t=d_2\Delta w+\gamma u-\delta w
\end{cases}
\end{align}
with $\chi, \xi, d_1, d_2, \alpha, \beta, \gamma, \delta>0$, 
which was introduced by Luca et al.~\cite{LCEM-2003} to describe 
the aggregation of microglial cells observed in Alzheimer's disease. 
The first mathematical work on this model was given by 
Tao--Wang~\cite{TW-2013} as will be described later. 
We can also refer to Jin--Wang~\cite{JW-2020} 
for the modeling and mathematical works on this model. 
On the other hand, the 
system~\eqref{P} is one of 
the chemotaxis models proposed by Keller--Segel~\cite{KS-1970} 
(for the variations with comprehensive studies, 
see Hillen--Painter~\cite{HP-2009}, 
Bellomo--Bellouquid--Tao--Winkler~\cite{BBTW-2015}  
and Arumugam--Tyagi~\cite{AT-2021}). 
Here {\it chemotaxis} is the property such that a species reacts on 
some chemical substance and moves towards or 
moves away from this substance. 
In this paper we are especially interested in the case of 
nonlinear diffusion and signal-dependent sensitivity; 
note that a quasilinear generalization of Keller--Segel systems such as 
\eqref{ARconst} was proposed by Painter--Hillen~\cite{HP-2002} 
to show the quorum effect in the chemotactic process. 
In particular, when the system has signal-dependent sensitivity, 
it is biologically meaningful in 
the Weber--Fechner law and it seems to be a mathematically challenging problem to study 
whether the solution remains bounded. 
%

%
%

%
\vspace{2mm}
We first focus on the Keller--Segel system 
with signal-dependent sensitivity,
%
\begin{align*}
\begin{cases}
      u_t=\Delta u
            -\nabla \cdot (u\chi(v)\nabla v),
    \\
      v_t=\Delta v-v+u,
\end{cases}
\end{align*}
%
where $\chi$ is a function. 
In this case global existence and boundedness 
were studied in \cite{A-2019, F-2015, FS-2016-1, FS-2018, FY-2014, Lan-2016, LW-2017, MY-2017, W-2010-2, W-2011, WY-2018}. 
More precisely, when $0<\chi(s) \le \frac{\chi_0}{(1+as)^k}$ for all $s>0$ 
with some $\chi_0>0$, $a>0$, $k>1$, 
Winkler \cite{W-2010-2} derived global existence and boundedness.
When $\chi(s)=\frac{\chi_0}{s}$ and $n \ge 2$, 
Winkler \cite{W-2011} showed global existence of classical solutions 
for $\chi_0<\sqrt{\frac{2}{n}}$, 
and global existence of weak solutions for $\chi_0<\sqrt{\frac{n+2}{3n-4}}$. 
On the other hand, when $0<\chi(s) \le \frac{\chi_0}{s^k}$ for all $s>0$ 
with some $\chi_0>0$, $k>1$ and $n \ge 2$, 
global existence and boundedness were obtained in \cite{FY-2014}. 
Also, Fujie \cite{F-2015} proved boundedness 
under the condition that $\chi(s)=\frac{\chi_0}{s}$ and 
$0<\chi_0<\sqrt{\frac{2}{n}}$. 
Moreover, in the two-dimensional setting, 
Lankeit \cite{Lan-2016} established boundedness 
if $\Omega$ is a convex domain and $\chi(s)=\frac{\chi_0}{s}$ 
for all $\chi_0 \in (0, \chi_0')$ with some $\chi_0'>1$. 
The condition for $\chi_0$ was relaxed by 
Lankeit--Winkler \cite{LW-2017} in a novel type of generalized solution framework. 
When $0 \le \chi(s) \le \frac{\chi_0}{(b+s)^k}$ 
for all $s>0$ with some small $\chi_0>0$ and 
$b \ge 0$, $k \ge 1$, 
boundedness of classical solutions was presented in~\cite{MY-2017}. 
Ahn~\cite{A-2019} improved the smallness condition for $\chi_0$ 
assumed in~\cite{MY-2017}, and showed 
stabilization (see also~\cite{WY-2018}). 
In the case that $\chi$ is a general function, 
global existence and boundedness of classical solutions 
were obtained in 
Fujie--Senba~\cite{FS-2016-1, FS-2018}. 
Particularly, in the two-dimensional setting, 
under the condition that 
$\chi>0$ fulfills $\lim_{s \to \infty}\chi(s)=0$, 
boundedness for small $\tau>0$ was shown in~\cite{FS-2016-1}. 
Moreover, when $\tau>0$ is sufficiently large, and 
$\chi$ satisfies that $\lim_{s \to \infty}\chi(s)=0$ if $n=2$ and 
$\limsup_{s \to \infty}s\chi(s)<\frac{n}{n-2}$ if $n \ge 3$, 
boundedness was proved 
in~\cite{FS-2018}. 
%

%
\vspace{2mm}
We next review the quasilinear Keller--Segel system with 
signal-dependent sensitivity,
%
\begin{align*}
\begin{cases}
      u_t=\nabla \cdot (D(u)\nabla u)
            -\nabla \cdot (G(u)\chi(v)\nabla v),
    \\[1.05mm]
      v_t=\Delta v-v+u,
\end{cases}
\end{align*}
%
where $D, G, \chi$ are functions. 
In the case that $\chi(s) \equiv 1$ and 
$\Omega$ is a convex domain, 
global existence and boundedness were showed 
in Tao--Winkler \cite{TW-2012} 
under the condition that 
$K_0(s+1)^{m-1} \le D(s) \le K_1(s+1)^{M-1}$ 
for all $s \ge 0$ with some $K_0, K_1>0$, $m, M \ge 1$ and 
$\frac{G(s)}{D(s)} \le K(s+1)^a$ 
for all $s\ge0$ with some $K>0$, $a<\frac{2}{n}$; 
note that the convexity of $\Omega$ 
was removed by \cite{ISY-2014}. 
On the other hand, when $\chi(s)=\frac{1}{s}$, under the condition that 
$K_0(s+1)^{m-1} \le D(s) \le K_1(s+1)^{M-1}$ 
for all $s \ge 0$ with $K_0, K_1>0$, $m, M \in \R$ and 
$\frac{G(s)}{D(s)} \le K(s+1)^a$ for all $s\ge0$ 
with some $K>0$, $a<\frac{2}{n}$, 
global existence and boundedness were established in \cite{FNY-2015}. 
However, the optimality of the condition $a<\frac{2}{n}$ 
was remained as an open problem. 
After that, by introducing a fractional type of test function, 
Ding \cite{D-2018} solved the open problem, that is, 
proved global existence and boundedness under 
the condition that $a<1$ and $m \le 1$.
Moreover, the problem in the case $m>1$ was solved by Jia--Yang \cite{JY-2019} 
under a differential condition. 
These mean that the signal-dependent sensitivity benefits 
global existence and boundedness of solutions. 
%

%
\vspace{2mm}
We now turn our eyes to the quasilinear attraction-repulsion chemotaxis system 
%
\begin{align*}
\begin{cases}
      u_t=\nabla \cdot (D(u)\nabla u)
            -\nabla \cdot (G(u)\chi(v)\nabla v)
            +\nabla\cdot(H(u)\xi(w)\nabla w),
    \\[1.05mm]
      \tau_1 v_t=\Delta v+\alpha u-\beta v,
    \\[1.05mm]
      \tau_2w_t=\Delta w+\gamma u-\delta w,
\end{cases}
\end{align*}
%
where $\alpha, \beta, \gamma, \delta>0$, $\tau_1, \tau_2 \ge 0$ and 
$D, G, H, \chi, \xi$ are functions. 
In the case that $D, \chi, \xi$ are constants 
and $G(s)=s$, $H(s)=s$ 
as well as $\tau_1=\tau_2=0$, Tao--Wang~\cite{TW-2013} established 
the first result on global existence and boundedness 
under the condition $\chi\alpha-\xi\gamma<0$; 
moreover, the authors proved finite-time blow-up 
by assuming $\chi\alpha-\xi\gamma>0$, $\beta=\delta$ and 
$\int_\Omega u_0>\frac{8\pi}{\chi\alpha-\xi\gamma}$ 
in the two-dimensional setting. 
Also, in the case that $D, \chi, \xi$ are constants 
and $G(s)=s$, $H(s)=s$ as well as $\tau_1=\tau_2=1$, 
Jin--Wang~\cite{JW-2020} derived
global existence and boundedness, and stabilization under 
the condition $\frac{\xi\gamma}{\chi\alpha} \ge C$ with some $C > 0$ 
in the two-dimensional setting. 
In addition, $D, \chi, \xi$ are constants and $\tau_1=\tau_2=1$, 
global existence and boundedness were studied in 
\cite{FS-2019, J-2015, LT-2015}; 
note that the transformation $z:=\chi v -\xi w$ is effective in this case. 
In the case that $\chi, \xi$ are constants, $\tau_1=1, \tau_2=0$ and $n \ge 2$, 
Lin--Mu--Gao \cite{LMG-2016} proved global existence and boundedness 
under the condition that $D(s)>0$ for all $s \ge 0$, 
$D(s) \ge \frac{a}{s^k}$ for all $s>0$ and all $k<\frac{2}{n}-1$ 
with some $a>0$ as well as $G(s)=s$, $H(s)=s$. 
Also, Li--Mu--Lin--Wang \cite{LMLW-2017} established global existence 
and boundedness under the following two conditions:
\begin{itemize}
\item[(i)] $\tau_1=\tau_2=0$, $G(s)=s$, $H(s)=s$ and 
$D(s) \ge as^{b-1}$ for all $s>0$ with some $a \in (0, 1]$, $b>2-\frac{2}{n}$
as well as $0<\chi(s) \le \frac{\chi_0}{s^k}$, $0<\xi(s)\le \frac{\xi_0}{s^\ell}$ 
for all $s>0$ with some $\chi_0, \xi_0>0$, $k \ge 1$, $\ell \ge 2$.
\item[(ii)] $\tau_1=1, \tau_2=0$, 
$G(s)=s$, $H(s)=s^r$ with $r \ge 2$ and 
$D(s) \ge as^{b-1}$ for all $s>0$ with some $a>0$, $b>r-\frac{2}{n}$ 
as well as $\chi(s) \equiv \chi_0$, $\xi(s)=\frac{\xi_0}{s}$ 
with some $\chi_0, \xi_0>0$.
\end{itemize}
In these literatures, the results were successfully obtained 
by using estimates for $v, w$ from below 
and by reducing the case that $\chi, \xi$ are constants. 
On the other hand, in the case of linear diffusion and normal sensitivity 
that $D(s) \equiv 1$, $G(s)=s$, $H(s)=s$, 
global existence and boundedness in the system with 
$\tau_1=\tau_2=1$ were proved in \cite{CMY-2020} 
by the method using a test function defined as a combination of 
an exponential function and integrals of $\chi, \xi$. 
However, since the proof in \cite{CMY-2020} strongly depends on  
$|(u+1)^{m-1}\nabla u|^2=(u+1)^{m-1}|\nabla u|^2$ 
(this holds true only in the case $m=1$!), 
the method does not work in the case 
$m \neq 1$. 
%

%
\vspace{2mm}
In summary, the system \eqref{P} with signal-dependent 
sensitivity has been studied in the following 
two restrictive cases: 
the first case that $D, G, H$ satisfy algebraic growth or decay conditions  
in the parabolic--elliptic--elliptic or parabolic--parabolic--elliptic version; 
the second case that $D(s) \equiv 1$, $G(s)=s$, $H(s)=s$ 
in the fully parabolic version. 
Especially,  recalling the case $w=0$, 
we know that the condition $a<\frac{2}{n}$ 
in \cite{FNY-2015} was removed 
in \cite{D-2018, JY-2019}. 
Therefore it is expected that even in \cite{LMLW-2017, LMG-2016}, 
the conditions are described without using 
the value $\frac{2}{n}$. 
The purpose of this paper is to establish 
global existence and boundedness of classical solutions to \eqref{P} 
under some conditions, independent of 
the dimension $n$, for algebraic growth or decay orders 
among $D, G, H$. 

%
\vspace{2mm}
In order to state the main theorem, 
we introduce conditions for the diffusivity $D$, 
the density-dependent sensitivities $G, H$ 
and the signal-dependent sensitivities $\chi, \xi$. 
We suppose that the functions $D, G, H$ satisfy
%
\begin{align}
    &D \in C^2([0, \infty)),\quad D(s)=a_0(s+1)^{m-1}
                                    \quad (a_0>0,\ m \in \R),\label{condiD}\\
    &G \in C^2([0, \infty)),\quad 0 \le G(s) \le b_0(s+1)^{q-1}
                                    \quad (b_0>0,\ q<\min\{2,\ m+1\}),\label{condiG}\\
    &H \in C^2([0, \infty)),\quad 0 \le H(s) \le c_0(s+1)^{r-1}
                                    \quad (c_0>0,\ r<\min\{2,\ m+1\}),\label{condiH}
\end{align}
%
and assume that the functions $\chi, \xi$ fulfill
%
\begin{align}
    &\chi \in C^{1+\vartheta_1}_{{\rm loc}}((0, \infty))\ 
      (0<\vartheta_1<1),\quad
      0\le\chi(s) \le \frac{\chi_0}{s^{k_1}}\quad 
      (\chi_0>0,\ k_1>1),\label{condichi}\\
    &\xi \in C^{1+\vartheta_2}_{{\rm loc}}((0, \infty))\ 
      (0<\vartheta_2<1),\quad
      0\le\xi(s) \le \frac{\xi_0}{s^{k_2}}\quad 
      (\xi_0>0,\ k_2>1).\label{condixi}
\end{align}
%
Then the main result reads as follows.
%
\begin{thm}\label{mainThm}
Let\/ $\Omega \subset \R^n$ $(n \ge 1)$ be a bounded domain 
with smooth boundary and let $d_1, d_2,  \alpha, \beta, \gamma, \delta>0$. 
Suppose that $D, G, H, \chi, \xi$ fulfill \eqref{condiD}--\eqref{condixi}. 
Then for all $(u_0, v_0, w_0)$ satisfying \eqref{condiu0}--\eqref{condiw0} 
there exists a unique triplet $(u, v, w)$ of nonnegative functions
%
\begin{align*}
    &u, v, w \in C^0(\cl{\Omega} \times [0, \infty)) 
                     \cap C^{2, 1}(\cl{\Omega} \times (0, \infty)),
\end{align*}
%
which solves \eqref{P} in the classical sense, and is bounded in the sense that
\begin{align*}
    &\|u(\cdot, t)\|_{L^\infty(\Omega)} \le C
\end{align*}
%
for all $t>0$ with some $C>0$. 
\end{thm}
%

\begin{cor}
Let $\xi=0$. 
Suppose that $D, G, \chi$ fulfill 
\eqref{condiD}, \eqref{condiG}, \eqref{condichi}, 
respectively. 
Then for all $(u_0, v_0)$ satisfying \eqref{condiu0} and \eqref{condiv0} 
there exists a unique classical solution $(u,v)$ 
which is bounded.
%
\end{cor}

\begin{remark}
The above corollary improves a previous result. 
Indeed, the condition 
$q<\frac{5-m}{2}$ ($m>1$) in \cite{JY-2019} is relaxed to $q<\min\{2,\ m+1\}$ ($m \in \R$). 
\end{remark}

%
The strategy for the proof of Theorem \ref{mainThm} is 
to show $L^p$-boundedness of $u$. 
In the case that $D(s) \equiv 1$, $G(s)=s$, $H(s)=s$, 
$L^p$-estimate for $u$ was established in \cite{CMY-2020} by deriving 
\begin{align*}
\frac{d}{dt}\int_\Omega u^p f(v, w)
\le c_1\int_\Omega u^p f(v, w)
      -c_2\Big(\int_\Omega u^p f(v, w)\Big)^{1+\vartheta}
\end{align*}
for some constants $c_1, c_2, \vartheta>0$ and 
some function $f \colon \R^2 \to \R$. 
Unfortunately, due to the nonlinearity of $D, G, H$, 
this method does not work in \eqref{P}. 
So, in this paper, we shift our method to that in 
\cite{D-2018} with the use of $\frac{d}{dt}\int_\Omega \frac{(u+1)^{p-m+1}}{v^{2k_1+\sigma_1-2}}$
with suitable $\sigma_1>0$.
However, even if $\frac{d}{dt}\int_\Omega \frac{(u+1)^{p-m+1}}{w^{2k_2+\sigma_2-2}}$ ($\sigma_2>0$) is added, the parallel method does not work, 
because some 
terms with the product of $v, w$ appear 
in the denominator. 
More precisely, combining $\frac{d}{dt}\int_\Omega (u+1)^{p-m+1}$, 
$\frac{d}{dt}\int_\Omega \frac{(u+1)^{p-m+1}}{v^{2k_1+\sigma_1-2}}$, 
$\frac{d}{dt}\int_\Omega \frac{(u+1)^{p-m+1}}{w^{2k_2+\sigma_2-2}}$, 
$\frac{L}{2d_1}\frac{d}{dt}\int_\Omega v^2$ and $\frac{M}{2d_2}\frac{d}{dt}\int_\Omega w^2$, 
we have several good terms such as 
\begin{align}\label{good}
-c_3\int_\Omega (u+1)^{p-2}|\nabla u|^2,\quad 
-c_4\int_\Omega \frac{(u+1)^{p-m-1}}{v^{2k_1+\sigma_1}}|\nabla v|^2, \quad 
-L\int_\Omega |\nabla v|^2
\end{align}
with some $c_3, c_4>0$ and sufficiently large $L>0$, 
and a lot of terms such as
\begin{align*}
&I_1:=\int_\Omega \frac{(u+1)^{p+q-m-2}}{v^{3k_1+\sigma_1-2}}|\nabla u||\nabla v|, \quad 
I_2:=\int_\Omega \frac{(u+1)^{p+q-\eta-3}}{v^{k_1}w^{2k_2+\sigma_2-2}}|\nabla u||\nabla v|.
\end{align*}
Using the estimate $v(x,t) \ge \mu_1$ 
with some $\mu_1>0$ (see \eqref{infv}), 
we can estimate $I_1$ as
\begin{align}
&I_1\le \ep_1\int_\Omega (u+1)^{p-2}|\nabla u|^2
        +\ep_2
          \int_\Omega \frac{(u+1)^{p-m+1}}
                             {v^{2k_1+\sigma_1}}
                            |\nabla v|^2
        +c_5
          \int_\Omega |\nabla v|^2 \label{strategyI1}
\end{align}
with small $\ep_1, \ep_2>0$ and some $c_5>0$, and hence all terms on the right-hand side of this inequality 
can be dominated by the good terms in \eqref{good}. 
On the other hand, using the estimate $w(x,t) \ge \mu_2$ 
with some $\mu_2>0$ (see \eqref{infw}), 
we can similarly estimate $I_2$ as 
\begin{align}
&I_2\le \ep_3\int_\Omega (u+1)^{p-2}|\nabla u|^2
        +\ep_4
          \int_\Omega \frac{(u+1)^{p-m+1}}
                             {v^{2k_1}}
                            |\nabla v|^2.
\end{align}
However, the second term on the right-hand side cannot be estimated 
by the good terms in \eqref{good}, 
because $\frac{1}{v^{2k_1}}\ (=\frac{v^{\sigma_1}}{v^{2k_1+\sigma_1}})$ cannot be estimated by 
$\frac{1}{v^{2k_1+\sigma_1}}$ due to the lack of the upper estimate for $v$. 
Thus we will overcome this difficulty by introducing a new test function 
with the product of $v, w$ in the denominator, 
that is, $\int_\Omega \frac{(u+1)^{p-\eta}}
                             {v^{2k_1+\sigma_3-2}w^{2k_2+\sigma_4-2}},$ 
where $\sigma_3, \sigma_4<0$, and $\eta>0$ will be fixed later. 
%

%
\vspace{2mm}
This paper is organized as follows. 
In Section \ref{Sec2} we collect some preliminary facts about 
local existence in \eqref{P}, the lower bounds for $v, w$, and  
the weighted Young inequality which will be employed frequently later. 
In Section \ref{Sec3} we mainly derive two differential inequalities  needed to 
prove global existence and boundedness (Theorem \ref{mainThm}).
%


\section{Preliminaries} \label{Sec2}

We first introduce a reasonable result on local existence of classical solutions to 
\eqref{P}, which can be proved by standard arguments 
based on the contraction mapping principle 
(see e.g., \cite{TW-2011} for nonlinear diffusion; 
\cite{FWY-2014} for signal-dependent sensitivity).

%
\begin{lem}\label{LoEx}
Let\/ $\Omega \subset \R^n$ $(n \ge 1)$ be a bounded domain 
with smooth boundary and let $d_1, d_2, \alpha, \beta, \gamma, \delta>0$. 
Assume that $D, G, H, \chi, \xi$ satisfy 
\eqref{condiD}--\eqref{condixi}. 
Then for all $(u_0, v_0, w_0)$ fulfilling \eqref{condiu0}--\eqref{condiw0} 
there exists $\tmax \in (0, \infty]$ such that 
\eqref{P} possesses a unique classical solution $(u, v, w)$ such that
%
\begin{align*}
    &u, v, w \in C^0(\cl{\Omega} \times [0, \tmax)) 
                     \cap C^{2, 1}(\cl{\Omega} \times (0, \tmax)).
\end{align*}
%
Moreover, 
%
\begin{align*}
   {\it if}\ \tmax<\infty,\quad 
   {\it then\ either}\ 
   \limsup_{t \nearrow \tmax}\|u(\cdot, t)\|_{L^\infty(\Omega)}
                  =\infty\ {\it or}\ 
\liminf_{t \nearrow \tmax}\inf_{x \in \Omega}
v(\cdot, t)=0,
\end{align*}
%
and $\int_\Omega u(\cdot, t)=\int_\Omega u_0$
%
%
for all $t \in (0, \tmax)$.
\end{lem}

In the following we suppose that $\Omega \subset \R^n$ $(n \ge 1)$ 
is a bounded domain with smooth boundary, 
$d_1, d_2, \alpha, \beta, \gamma, \delta>0$ and $D, G, H, \chi, \xi$ fulfill 
\eqref{condiD}--\eqref{condixi} as well as $(u_0, v_0, w_0)$ satisfies 
\eqref{condiu0}--\eqref{condiw0}. 
Then we denote by $(u, v, w)$ the local classical solution of \eqref{P} 
given in Lemma \ref{LoEx} and by $\tmax$ its maximal existence time. 
We next present the result on the lower bounds for $v, w$, 
which was obtained in \cite[Lemma 2.2]{F-2015} 
(see also \cite[Lemma 2.1 and Remark 2.2]{MY-2017}).

%
\begin{lem}\label{LB}
Assume that $(u, v, w)$ is the local classical solution of \eqref{P}. 
Then there exist constants $\mu_1, \mu_2>0$ such that
%
\begin{align}
   &\inf_{x \in \Omega} v(x, t) \ge \mu_1,\label{infv}\\
   &\inf_{x \in \Omega} w(x, t) \ge \mu_2\label{infw}
\end{align}
%
for all $t \in (0, \tmax)$.
\end{lem}
%

We finally recall the following weighted Young inequality 
which will be used frequently later.

%
\begin{lem}\label{Young}
Let $p, q \in (1, \infty)$ satisfy $\frac{1}{p}+\frac{1}{q}=1$. 
Then for all $a, b \ge 0$ and all $\ep>0$, the inequality
%
\begin{align*}
    ab \le \ep^p\,\frac{a^p}{p}+\frac{1}{\ep^q}\frac{b^q}{q}
\end{align*}
%
holds.
\end{lem}
%


\section{Proof of Theorem \ref{mainThm}} \label{Sec3}

In this section we mainly derive two differential inequalities 
which lead to $L^p$-estimate for $u$. 
The first one is given by the following lemma. 

%
\begin{lem}\label{difineq1}
Assume that 
%
\begin{align}
    &p>\max\{1,\ m,\  2(m-q+1),\ 2(m-r+1)\},\label{condip}\\
    &\eta<\min\{2(m-q+1),\ 2(m-r+1)\}.\label{condieta}
\end{align}
%
Then for all $\ep_{01}, \ep_{02}>0$ there exist constants
$C_1=C_1(b_0, p, q, m, k_1, \mu_1, \ep_{01})>0$ and 
$C_2=C_2(c_0, p, r, m, k_2, \mu_2, \ep_{02})>0$ such that
%
\begin{align}\label{DI1}
    &\frac{d}{dt}\int_\Omega (u+1)^{p-m+1}
        +\frac{a_0(p-m)(p-m+1)}{2}\int_\Omega (u+1)^{p-2}|\nabla u|^2
        \notag\\
    &\quad\,
        \le \ep_{01}\int_\Omega\frac{|\nabla v|^2}{v^{2k_1+\sigma_1}}(u+1)^{p-\eta}
        +\ep_{02}\int_\Omega\frac{|\nabla w|^2}{w^{2k_2+\sigma_2}}(u+1)^{p-\eta}
        \notag\\
    &\quad\,\quad\,
        +C_1\int_\Omega |\nabla v|^2+C_2\int_\Omega |\nabla w|^2
\end{align}
%
for all $t \in (0, \tmax)$, where 
$\sigma_1:=\frac{2k_1(2m-2q-\eta+2)}{p-2(m-q+1)}>0$ and 
$\sigma_2:=\frac{2k_2(2m-2r-\eta+2)}{p-2(m-r+1)}>0$.
\end{lem}
%

\begin{proof}
Straightforward calculations, integration by parts 
and \eqref{condiD}--\eqref{condixi} yield that
%
\begin{align}\label{DI1-1}
    &\frac{d}{dt}\int_\Omega (u+1)^{p-m+1}\notag\\
    &\quad\,
      =(p-m+1)\int_\Omega (u+1)^{p-m}
        \nabla \cdot [D(u)\nabla u-G(u)\chi(v)\nabla v+H(u)\xi(w)\nabla w]
        \notag\\
    &\quad\,
      ={}-(p-m)(p-m+1)\int_\Omega D(u)(u+1)^{p-m-1}|\nabla u|^2\notag\\
    &\qquad\ \,
        +(p-m)(p-m+1)\int_\Omega G(u)(u+1)^{p-m-1}\chi(v)
          \nabla u \cdot \nabla v\notag\\
    &\qquad\,\ 
        -(p-m)(p-m+1)\int_\Omega H(u)(u+1)^{p-m-1}\xi(w)
          \nabla u \cdot \nabla w\notag\\
    &\quad\,
      \le {}-a_0(p-m)(p-m+1)\int_\Omega (u+1)^{p-2}|\nabla u|^2\notag\\
    &\qquad\,\ 
        +b_0(p-m)(p-m+1)\int_\Omega (u+1)^{p+q-m-2}\frac{\chi_0}{v^{k_1}}
          |\nabla u||\nabla v|\notag\\
    &\qquad\,\ 
        +c_0(p-m)(p-m+1)\int_\Omega (u+1)^{p+r-m-2}\frac{\xi_0}{w^{k_2}}
          |\nabla u||\nabla w|, 
\end{align}
%
where we used the fact $p>m$ (see \eqref{condip}). 
We now estimate the second and third terms 
on the rightmost summand of \eqref{DI1-1}. 
Using Lemma \ref{Young}, we have
%
\begin{align}\label{DI1-2}
    &\int_\Omega (u+1)^{p+q-m-2}\frac{\chi_0}{v^{k_1}}
          |\nabla u||\nabla v|\notag\\
    &\quad\,
      =\int_\Omega \frac{1}{2}\Big(\frac{a_0}{b_0}\Big)^{\frac{1}{2}}
                           (u+1)^\frac{p-2}{2}|\nabla u| \cdot 
                           2\chi_0\Big(\frac{b_0}{a_0}\Big)^{\frac{1}{2}}
                           \frac{|\nabla v|}{v^{k_1}}(u+1)^{\frac{p-2(m-q+1)}{2}}
      \notag\\
    &\quad\,
      \le \frac{a_0}{4b_0}\int_\Omega (u+1)^{p-2}|\nabla u|^2
           +\frac{b_0\chi_0^2}{a_0}\int_\Omega
             \frac{|\nabla v|^2}{v^{2k_1}}(u+1)^{p-2(m-q+1)}.
\end{align}
%
Since $\theta:=\frac{p-\eta}{p-2(m-q+1)}>1$ due to \eqref{condieta}, 
it follows from Lemma \ref{Young} and \eqref{infv} that 
for all $\widetilde{\ep_{01}}>0$, 
%
\begin{align}\label{DI1-3}
    &\frac{b_0\chi_0^2}{a_0}\int_\Omega
             \frac{|\nabla v|^2}{v^{2k_1}}(u+1)^{p-2(m-q+1)}
      \notag\\
    &\quad\,
      =\frac{b_0\chi_0^2}{a_0}\int_\Omega
             \frac{|\nabla v|^\frac{2}{\theta}}{v^{\frac{2k_1+\sigma_1}{\theta}}}
             (u+1)^{p-2(m-q+1)} \cdot
             \frac{|\nabla v|^{2-\frac{2}{\theta}}}
                    {v^{2k_1-\frac{2k_1+\sigma_1}{\theta}}}
      \notag\\
    &\quad\,
      \le \widetilde{\ep_{01}}\int_\Omega\frac{|\nabla v|^2}{v^{2k_1+\sigma_1}}
           (u+1)^{p-\eta}
           +c_1\int_\Omega |\nabla v|^2
\end{align}
%
holds with $\sigma_1:=\frac{2k_1(2m-2q-\eta+2)}{p-2(m-q+1)}>0$ 
and $c_1=c_1(p, q, m, k_1, \mu_1, \widetilde{\ep_{01}})>0$. 
Thus, combining \eqref{DI1-2} and \eqref{DI1-3}, we see that
%
\begin{align}\label{DI1-4}
    &\int_\Omega (u+1)^{p+q-m-2}\frac{\chi_0}{v^{k_1}}
          |\nabla u||\nabla v|\notag\\
    &\quad\,
      \le \frac{a_0}{4b_0}\int_\Omega (u+1)^{p-2}|\nabla u|^2
           +\widetilde{\ep_{01}}\int_\Omega\frac{|\nabla v|^2}{v^{2k_1+\sigma_1}}
           (u+1)^{p-\eta}
           +c_1\int_\Omega |\nabla v|^2.
\end{align}
%
Similarly, we obtain that for all $\widetilde{\ep_{02}}>0$, 
%
\begin{align}\label{DI1-5}
    &\int_\Omega (u+1)^{p+r-m-2}\frac{\xi_0}{w^{k_2}}
          |\nabla u||\nabla w|\notag\\
    &\quad\,
      \le \frac{a_0}{4c_0}\int_\Omega (u+1)^{p-2}|\nabla u|^2
           +\widetilde{\ep_{02}}\int_\Omega\frac{|\nabla w|^2}{w^{2k_2+\sigma_2}}
           (u+1)^{p-\eta}
           +c_2\int_\Omega |\nabla w|^2
\end{align}
%
holds with $\sigma_2:=\frac{2k_2(2m-2r-\eta+2)}{p-2(m-r+1)}>0$ 
and $c_2=c_2(p, r, m, k_2, \mu_2, \widetilde{\ep_{02}})>0$. 
Hence a combination of \eqref{DI1-1}, \eqref{DI1-4} and \eqref{DI1-5} 
yields that 
%
\begin{align*}
    \frac{d}{dt}\int_\Omega (u+1)^{p-m+1}
    &\le {}-a_0(p-m)(p-m+1)\int_\Omega (u+1)^{p-2}|\nabla u|^2\\
    &\quad\ 
        +\frac{a_0(p-m)(p-m+1)}{4}\int_\Omega (u+1)^{p-2}|\nabla u|^2\notag\\
    &\quad\ 
        +b_0(p-m)(p-m+1)\widetilde{\ep_{01}}
           \int_\Omega\frac{|\nabla v|^2}{v^{2k_1+\sigma_1}}(u+1)^{p-\eta}
           \notag\\
    &\quad\ 
           +b_0(p-m)(p-m+1)c_1\int_\Omega |\nabla v|^2\notag\\
    &\quad\ 
        +\frac{a_0(p-m)(p-m+1)}{4}\int_\Omega (u+1)^{p-2}|\nabla u|^2\notag\\
    &\quad\ 
        +c_0(p-m)(p-m+1)\widetilde{\ep_{02}}
           \int_\Omega\frac{|\nabla w|^2}{w^{2k_2+\sigma_2}}(u+1)^{p-\eta}
           \notag\\
    &\quad\ 
         +c_0(p-m)(p-m+1)c_2\int_\Omega |\nabla w|^2\notag\\
    &
      ={}-\frac{a_0(p-m)(p-m+1)}{2}\int_\Omega (u+1)^{p-2}|\nabla u|^2\notag\\
    &\quad\ 
        +b_0(p-m)(p-m+1)\widetilde{\ep_{01}}
           \int_\Omega\frac{|\nabla v|^2}{v^{2k_1+\sigma_1}}(u+1)^{p-\eta}
           \notag\\
    &\quad\ 
        +c_0(p-m)(p-m+1)\widetilde{\ep_{02}}
           \int_\Omega\frac{|\nabla w|^2}{w^{2k_2+\sigma_2}}(u+1)^{p-\eta}
           \notag\\
    &\quad\ 
           +c_3\int_\Omega |\nabla v|^2+c_4\int_\Omega |\nabla w|^2
\end{align*}
%
with $c_3:=b_0(p-m)(p-m+1)c_1>0$ and $c_4:=c_0(p-m)(p-m+1)c_2>0$. 
Therefore we have
%
\begin{align*}
    &\frac{d}{dt}\int_\Omega (u+1)^{p-m+1}
       +\frac{a_0(p-m)(p-m+1)}{2}\int_\Omega (u+1)^{p-2}|\nabla u|^2\notag\\
    &\quad\, \le
       b_0(p-m)(p-m+1)\widetilde{\ep_{01}}
           \int_\Omega\frac{|\nabla v|^2}{v^{2k_1+\sigma_1}}(u+1)^{p-\eta}
           \notag\\
    &\qquad\,\,
        +c_0(p-m)(p-m+1)\widetilde{\ep_{02}}
           \int_\Omega\frac{|\nabla w|^2}{w^{2k_2+\sigma_2}}(u+1)^{p-\eta}
           \notag\\
    &\qquad\,\,
           +c_3\int_\Omega |\nabla v|^2+c_4\int_\Omega |\nabla w|^2,
\end{align*}
%
which leads to \eqref{DI1} due to  arbitrariness of 
$\widetilde{\ep_{01}}, \widetilde{\ep_{02}}>0$. 
\end{proof}

The second inequality to be shown is given by the following lemma.

%
\begin{lem}\label{difineq2}
Assume that $q<\min\{2,\ m+1\}$, $r<\min\{2,\ m+1\}$, $k_1>1$, $k_2>1$, 
$2(1-k_1)<\sigma_3<0$, $2(1-k_2)<\sigma_4<0$ and
%
\begin{align}\label{condip2}
    &p>\max\Big\{\,
         1,\ m,\ 2(m-q+1),\ 2(m-r+1),\notag\\
    &\hspace{19mm}
         \frac{2[(m-1)(2k_1+\sigma_i-1)+(m-\eta-1)]}{2k_1+\sigma_i-2}\ 
         (i\in\{1,3\}),\notag\\
    &\hspace{19mm} 
         \frac{2[(m-1)(2k_2+\sigma_j-1)+(m-\eta-1)]}{2k_2+\sigma_j-2}\ 
         (j\in\{2,4\})\,
    \Big\}
\end{align}
%
as well as
%
\begin{align}
    \label{condieta2}
    \max\{2(m-1),\ 0\}<\eta<\min\{2(m-q+1),\ 2(m-r+1)\}, 
\end{align}
%
where $\sigma_1, \sigma_2$ in \eqref{condip2} are defined in 
{\rm Lemma \ref{difineq1};}
note that existence of $\eta$ satisfying \eqref{condieta2} 
is guaranteed by $q, r<\min\{2,\ m+1\}$. 
Then there exist constants $\ep_{03}>0$ and $C_k>0$ $(k \in \{3, 4, 5, 6, 7, 8\})$ 
such that
%
\begin{align}\label{DI2}
    &\quad\,\,\,
      \ep_{03}\frac{d}{dt}\int_\Omega \frac{(u+1)^{p-\eta}}
                             {v^{2k_1+\sigma_3-2}w^{2k_2+\sigma_4-2}}
      +\frac{d}{dt}\int_\Omega \frac{(u+1)^{p-\eta}}
                             {v^{2k_1+\sigma_1-2}}
      +\frac{d}{dt}\int_\Omega \frac{(u+1)^{p-\eta}}
                             {w^{2k_2+\sigma_2-2}}
      \notag\\
    &\quad\,
      +C_3\int_\Omega\frac{|\nabla v|^2}{v^{2k_1+\sigma_1}}(u+1)^{p-\eta}
      +C_4\int_\Omega\frac{|\nabla w|^2}{w^{2k_2+\sigma_2}}(u+1)^{p-\eta}
      \notag\\
    &\le 
      \frac{a_0(p-m)(p-m+1)}{4}\int_\Omega (u+1)^{p-2}|\nabla u|^2
      +C_5\int_\Omega |\nabla v|^2+C_6\int_\Omega |\nabla w|^2
      \notag\\
    &\quad\,
      +C_7\int_\Omega (u+1)^p+C_8
\end{align}
%
for all $t \in (0, \tmax)$.
\end{lem}
%

\begin{proof}
We first estimate 
$\frac{d}{dt}\int_\Omega \frac{(u+1)^{p-\eta}}{v^{2k_1+\sigma_3-2}w^{2k_2+\sigma_4-2}}$. 
Using the equations in \eqref{P}, 
we have
%
\begin{align}\label{DI2-1}
    &\frac{d}{dt}\int_\Omega 
      \frac{(u+1)^{p-\eta}}
             {v^{2k_1+\sigma_3-2}w^{2k_2+\sigma_4-2}}\notag\\
    &\quad\, 
      =(p-\eta)\int_\Omega 
        \frac{(u+1)^{p-\eta-1}}
                {v^{2k_1+\sigma_3-2}w^{2k_2+\sigma_4-2}}u_t\notag\\
    &\qquad\,\,\, 
      -(2k_1+\sigma_3-2)\int_\Omega 
        \frac{(u+1)^{p-\eta}}{v^{2k_1+\sigma_3-1}w^{2k_2+\sigma_4-2}}v_t
      \notag\\
    &\qquad\,\,\, 
      -(2k_2+\sigma_4-2)\int_\Omega 
        \frac{(u+1)^{p-\eta}}{v^{2k_1+\sigma_3-2}w^{2k_2+\sigma_4-1}}w_t
      \notag\\
    &\quad\, 
      =(p-\eta)\int_\Omega 
        \frac{(u+1)^{p-\eta-1}}
                {v^{2k_1+\sigma_3-2}w^{2k_2+\sigma_4-2}}
        \nabla \cdot [D(u)\nabla u-G(u)\chi(v)\nabla v+H(u)\xi(w)\nabla w]
        \notag\\
    &\qquad\,\,\, 
      -(2k_1+\sigma_3-2)\int_\Omega 
        \frac{(u+1)^{p-\eta}}{v^{2k_1+\sigma_3-1}w^{2k_2+\sigma_4-2}}
        (d_1\Delta v+\alpha u-\beta v)
      \notag\\
    &\qquad\,\,\, 
      -(2k_2+\sigma_4-2)\int_\Omega 
        \frac{(u+1)^{p-\eta}}{v^{2k_1+\sigma_3-2}w^{2k_2+\sigma_4-1}}
        (d_2\Delta w+\gamma u-\delta w)
      \notag\\
    &\quad\,
       =:(p-\eta)J_1+(2k_1+\sigma_3-2)J_2+(2k_2+\sigma_4-2)J_3.
\end{align}
%
As to the first term $J_1$, integrating by parts leads to
%
\begin{align*}
    J_1
    &=-\int_\Omega 
        \nabla\Big(\frac{(u+1)^{p-\eta-1}}
                {v^{2k_1+\sigma_3-2}w^{2k_2+\sigma_4-2}}\Big)
        \cdot [D(u)\nabla u-G(u)\chi(v)\nabla v+H(u)\xi(w)\nabla w]\notag\\[2mm]
    &=-(p-\eta-1)\int_\Omega 
        \frac{(u+1)^{p-\eta-2}\nabla u}
                {v^{2k_1+\sigma_3-2}w^{2k_2+\sigma_4-2}}
        \cdot [D(u)\nabla u-G(u)\chi(v)\nabla v+H(u)\xi(w)\nabla w]\notag\\
    &\quad\,
      +(2k_1+\sigma_3-2)\int_\Omega 
        \frac{(u+1)^{p-\eta-1}\nabla v}
                {v^{2k_1+\sigma_3-1}w^{2k_2+\sigma_4-2}}
        \cdot [D(u)\nabla u-G(u)\chi(v)\nabla v +H(u)\xi(w)\nabla w]\notag\\
    &\quad\,
      +(2k_2+\sigma_4-2)\int_\Omega 
        \frac{(u+1)^{p-\eta-1}\nabla w}
                {v^{2k_1+\sigma_3-2}w^{2k_2+\sigma_4-1}}
        \cdot [D(u)\nabla u-G(u)\chi(v)\nabla v+H(u)\xi(w)\nabla w]\notag\\[2mm]
    &=-(p-\eta-1)\int_\Omega 
        \frac{(u+1)^{p-\eta-2}}
                {v^{2k_1+\sigma_3-2}w^{2k_2+\sigma_4-2}}
        [D(u)|\nabla u|^2-G(u)\chi(v)\nabla u\cdot \nabla v\notag\\
    &\hspace{110mm}
      +H(u)\xi(w)\nabla u \cdot \nabla w]
        \notag\\
    &\quad\,
      +(2k_1+\sigma_3-2)\int_\Omega 
        \frac{(u+1)^{p-\eta-1}}
                {v^{2k_1+\sigma_3-1}w^{2k_2+\sigma_4-2}}
        [D(u)\nabla u \cdot \nabla v-G(u)\chi(v)|\nabla v|^2\notag\\
    &\hspace{110mm}
      +H(u)\xi(w)\nabla v \cdot \nabla w]
        \notag\\
    &\quad\,
      +(2k_2+\sigma_4-2)\int_\Omega 
        \frac{(u+1)^{p-\eta-1}}
                {v^{2k_1+\sigma_3-2}w^{2k_2+\sigma_4-1}}
        [D(u)\nabla u \cdot \nabla w
         -G(u)\chi(v)\nabla v\cdot \nabla w\notag\\
    &\hspace{110mm}
      +H(u)\xi(w)|\nabla w|^2], 
\end{align*}
and then using \eqref{condiD}--\eqref{condiH} yields
\begin{align*}
    J_1
       &\le 
      -a_0(p-\eta-1)\int_\Omega 
        \frac{(u+1)^{p+m-\eta-3}}
                {v^{2k_1+\sigma_3-2}w^{2k_2+\sigma_4-2}}|\nabla u|^2\notag\\
    &\quad\,
      +b_0\chi_0(p-\eta-1)\int_\Omega 
        \frac{(u+1)^{p+q-\eta-3}}
                {v^{3k_1+\sigma_3-2}w^{2k_2+\sigma_4-2}}|\nabla u||\nabla v|
      \notag\\
    &\quad\,
      +c_0\xi_0(p-\eta-1)\int_\Omega 
        \frac{(u+1)^{p+r-\eta-3}}
                {v^{2k_1+\sigma_3-2}w^{3k_2+\sigma_4-2}}|\nabla u||\nabla w|
      \notag\\
    &\quad\,
      +a_0(2k_1+\sigma_3-2)\int_\Omega 
        \frac{(u+1)^{p+m-\eta-2}}
                {v^{2k_1+\sigma_3-1}w^{2k_2+\sigma_4-2}}|\nabla u||\nabla v|
      \notag\\
    &\quad\,
      -(2k_1+\sigma_3-2)\int_\Omega 
        \frac{(u+1)^{p-\eta-1}}
                {v^{2k_1+\sigma_3-1}w^{2k_2+\sigma_4-2}}G(u)\chi(v)|\nabla v|^2
      \notag\\
    &\quad\,
      +c_0\xi_0(2k_1+\sigma_3-2)\int_\Omega 
        \frac{(u+1)^{p+r-\eta-2}}
                {v^{2k_1+\sigma_3-1}w^{3k_2+\sigma_4-2}}|\nabla v||\nabla w|
      \notag\\
    &\quad\,
      +a_0(2k_2+\sigma_4-2)\int_\Omega 
        \frac{(u+1)^{p+m-\eta-2}}
                {v^{2k_1+\sigma_3-2}w^{2k_2+\sigma_4-1}}|\nabla u||\nabla w|
      \notag\\
    &\quad\,
      +b_0\chi_0(2k_2+\sigma_4-2)\int_\Omega 
        \frac{(u+1)^{p+q-\eta-2}}
                {v^{3k_1+\sigma_3-2}w^{2k_2+\sigma_4-1}}|\nabla v||\nabla w|
      \notag\\
    &\quad\,
      +c_0\xi_0(2k_2+\sigma_4-2)\int_\Omega 
        \frac{(u+1)^{p+r-\eta-2}}
                {v^{2k_1+\sigma_3-2}w^{3k_2+\sigma_4-1}}|\nabla w|^2.
\end{align*}
%
As to the second term $J_2$ and third term $J_3$, 
due to integration by parts and straightforward calculations,  
we infer
%
\begin{align*}
    J_2
    &=d_1\int_\Omega \nabla\Big(
        \frac{(u+1)^{p-\eta}}{v^{2k_1+\sigma_3-1}w^{2k_2+\sigma_4-2}}
        \Big)\cdot\nabla v\notag\\
    &\quad\,
      -\alpha\int_\Omega 
        \frac{u(u+1)^{p-\eta}}{v^{2k_1+\sigma_3-1}w^{2k_2+\sigma_4-2}}
      +\beta\int_\Omega
        \frac{(u+1)^{p-\eta}}{v^{2k_1+\sigma_3-2}w^{2k_2+\sigma_4-2}}
      \notag\\
    &=
      d_1(p-\eta)\int_\Omega 
      \frac{(u+1)^{p-\eta-1}}{v^{2k_1+\sigma_3-1}w^{2k_2+\sigma_4-2}}
      \nabla u \cdot \nabla v
      \notag\\
    &\quad\,
      -d_1(2k_1+\sigma_3-1)\int_\Omega 
      \frac{(u+1)^{p-\eta}}{v^{2k_1+\sigma_3}w^{2k_2+\sigma_4-2}}|\nabla v|^2
      \notag\\
    &\quad\,
      -d_1(2k_2+\sigma_4-2)\int_\Omega 
      \frac{(u+1)^{p-\eta}}{v^{2k_1+\sigma_3-1}w^{2k_2+\sigma_4-1}}
      \nabla v\cdot\nabla w
      \notag\\
    &\quad\,
      -\alpha\int_\Omega 
        \frac{u(u+1)^{p-\eta}}{v^{2k_1+\sigma_3-1}w^{2k_2+\sigma_4-2}}
      +\beta\int_\Omega
        \frac{(u+1)^{p-\eta}}{v^{2k_1+\sigma_3-2}w^{2k_2+\sigma_4-2}}
    \notag\\
    &\le
      d_1(p-\eta)\int_\Omega 
      \frac{(u+1)^{p-\eta-1}}{v^{2k_1+\sigma_3-1}w^{2k_2+\sigma_4-2}}
      |\nabla u||\nabla v|
      \notag\\
    &\quad\,
      -d_1(2k_1+\sigma_3-1)\int_\Omega 
      \frac{(u+1)^{p-\eta}}{v^{2k_1+\sigma_3}w^{2k_2+\sigma_4-2}}|\nabla v|^2
      \notag\\
    &\quad\,
      +d_1(2k_2+\sigma_4-2)\int_\Omega 
      \frac{(u+1)^{p-\eta}}{v^{2k_1+\sigma_3-1}w^{2k_2+\sigma_4-1}}
      |\nabla v||\nabla w|
      \notag\\
    &\quad\,
      -\alpha\int_\Omega 
        \frac{u(u+1)^{p-\eta}}{v^{2k_1+\sigma_3-1}w^{2k_2+\sigma_4-2}}
      +\beta\int_\Omega
        \frac{(u+1)^{p-\eta}}{v^{2k_1+\sigma_3-2}w^{2k_2+\sigma_4-2}}
\end{align*}
%
and 
%
\begin{align*}
    J_3
    &=
      d_2(p-\eta)\int_\Omega 
      \frac{(u+1)^{p-\eta-1}}{v^{2k_1+\sigma_3-2}w^{2k_2+\sigma_4-1}}
      \nabla u \cdot \nabla w
      \notag\\
    &\quad\,
      -d_2(2k_1+\sigma_3-2)\int_\Omega 
      \frac{(u+1)^{p-\eta}}{v^{2k_1+\sigma_3-1}w^{2k_2+\sigma_4-1}}
      \nabla v \cdot \nabla w
      \notag\\
    &\quad\,
      -d_2(2k_2+\sigma_4-1)\int_\Omega 
      \frac{(u+1)^{p-\eta}}{v^{2k_1+\sigma_3-2}w^{2k_2+\sigma_4}}
      |\nabla w|^2
      \notag\\
    &\quad\,
      -\gamma\int_\Omega 
        \frac{u(u+1)^{p-\eta}}{v^{2k_1+\sigma_3-2}w^{2k_2+\sigma_4-1}}
      +\delta\int_\Omega
        \frac{(u+1)^{p-\eta}}{v^{2k_1+\sigma_3-2}w^{2k_2+\sigma_4-2}}
      \notag\\
    &\le
      d_2(p-\eta)\int_\Omega 
      \frac{(u+1)^{p-\eta-1}}{v^{2k_1+\sigma_3-2}w^{2k_2+\sigma_4-1}}
      |\nabla u||\nabla w|
      \notag\\
    &\quad\,
      +d_2(2k_1+\sigma_3-2)\int_\Omega 
      \frac{(u+1)^{p-\eta}}{v^{2k_1+\sigma_3-1}w^{2k_2+\sigma_4-1}}
      |\nabla v||\nabla w|
      \notag\\
    &\quad\,
      -d_2(2k_2+\sigma_4-1)\int_\Omega 
      \frac{(u+1)^{p-\eta}}{v^{2k_1+\sigma_3-2}w^{2k_2+\sigma_4}}
      |\nabla w|^2
      \notag\\
    &\quad\,
      -\gamma\int_\Omega 
        \frac{u(u+1)^{p-\eta}}{v^{2k_1+\sigma_3-2}w^{2k_2+\sigma_4-1}}
      +\delta\int_\Omega
        \frac{(u+1)^{p-\eta}}{v^{2k_1+\sigma_3-2}w^{2k_2+\sigma_4-2}}.
\end{align*}
%
Combining \eqref{DI2-1} and the above estimates for $J_1, J_2, J_3$, 
we obtain
%
\begin{align}\label{DI2-5}
    &\frac{d}{dt}\int_\Omega 
      \frac{(u+1)^{p-\eta}}
             {v^{2k_1+\sigma_3-2}w^{2k_2+\sigma_4-2}}\notag\\
    &\quad\, 
      \le
      -A_1\int_\Omega 
        \frac{(u+1)^{p+m-\eta-3}}
                {v^{2k_1+\sigma_3-2}w^{2k_2+\sigma_4-2}}|\nabla u|^2   
      +A_2\int_\Omega 
        \frac{(u+1)^{p+q-\eta-3}}
                {v^{3k_1+\sigma_3-2}w^{2k_2+\sigma_4-2}}|\nabla u||\nabla v|
      \notag\\
    &\quad\,\quad\,
      +A_3\int_\Omega 
        \frac{(u+1)^{p+r-\eta-3}}
                {v^{2k_1+\sigma_3-2}w^{3k_2+\sigma_4-2}}|\nabla u||\nabla w|
      +A_4\int_\Omega 
        \frac{(u+1)^{p+m-\eta-2}}
                {v^{2k_1+\sigma_3-1}w^{2k_2+\sigma_4-2}}|\nabla u||\nabla v|
      \notag\\
    &\quad\,\quad\,
      -A_5\int_\Omega 
        \frac{(u+1)^{p-\eta-1}}
                {v^{2k_1+\sigma_3-1}w^{2k_2+\sigma_4-2}}G(u)\chi(v)|\nabla v|^2
      +A_6\int_\Omega 
        \frac{(u+1)^{p+r-\eta-2}}
                {v^{2k_1+\sigma_3-1}w^{3k_2+\sigma_4-2}}|\nabla v||\nabla w|
      \notag\\
    &\quad\,\quad\,
      +A_7\int_\Omega 
        \frac{(u+1)^{p+m-\eta-2}}
                {v^{2k_1+\sigma_3-2}w^{2k_2+\sigma_4-1}}|\nabla u||\nabla w|
      +A_8\int_\Omega 
        \frac{(u+1)^{p+q-\eta-2}}
                {v^{3k_1+\sigma_3-2}w^{2k_2+\sigma_4-1}}|\nabla v||\nabla w|
      \notag\\
    &\quad\,\quad\,
      +A_9\int_\Omega 
        \frac{(u+1)^{p+r-\eta-2}}
                {v^{2k_1+\sigma_3-2}w^{3k_2+\sigma_4-1}}|\nabla w|^2\notag\\
    &\quad\,\quad\,
      +A_{10}\int_\Omega 
      \frac{(u+1)^{p-\eta-1}}{v^{2k_1+\sigma_3-1}w^{2k_2+\sigma_4-2}}
      |\nabla u||\nabla v|
      -A_{11}\int_\Omega 
      \frac{(u+1)^{p-\eta}}{v^{2k_1+\sigma_3}w^{2k_2+\sigma_4-2}}|\nabla v|^2
      \notag\\
    &\quad\,\quad\,
      +A_{12}\int_\Omega 
      \frac{(u+1)^{p-\eta}}{v^{2k_1+\sigma_3-1}w^{2k_2+\sigma_4-1}}
      |\nabla v||\nabla w|
      -A_{13}\int_\Omega 
        \frac{u(u+1)^{p-\eta}}{v^{2k_1+\sigma_3-1}w^{2k_2+\sigma_4-2}}
      \notag\\
    &\quad\,\quad\,
      +A_{14}\int_\Omega
        \frac{(u+1)^{p-\eta}}{v^{2k_1+\sigma_3-2}w^{2k_2+\sigma_4-2}}
      \notag\\
    &\quad\,\quad\,
      +A_{15}\int_\Omega 
      \frac{(u+1)^{p-\eta-1}}{v^{2k_1+\sigma_3-2}w^{2k_2+\sigma_4-1}}
      |\nabla u||\nabla w|
      +A_{16}\int_\Omega 
      \frac{(u+1)^{p-\eta}}{v^{2k_1+\sigma_3-1}w^{2k_2+\sigma_4-1}}
      |\nabla v||\nabla w|
      \notag\\
    &\quad\,\quad\,
      -A_{17}\int_\Omega 
      \frac{(u+1)^{p-\eta}}{v^{2k_1+\sigma_3-2}w^{2k_2+\sigma_4}}
      |\nabla w|^2
      -A_{18}\int_\Omega 
        \frac{u(u+1)^{p-\eta}}{v^{2k_1+\sigma_3-2}w^{2k_2+\sigma_4-1}}
      \notag\\
    &\quad\,\quad\,
      +A_{19}\int_\Omega
        \frac{(u+1)^{p-\eta}}{v^{2k_1+\sigma_3-2}w^{2k_2+\sigma_4-2}},
\end{align}
%
where
%
\begin{alignat*}{3}
        A_1 &:=a_0(p-\eta-1)(p-\eta), & 
             &\qquad & 
        A_2 &:=b_0\chi_0(p-\eta-1)(p-\eta),
    \\ 
        A_3 &:=c_0\xi_0(p-\eta-1)(p-\eta), &
             &\qquad & 
        A_4 &:=a_0(p-\eta)(2k_1+\sigma_3-2),
    \\
        A_5 &:=(p-\eta)(2k_1+\sigma_3-2), &
             &\qquad &
        A_6 &:=c_0\xi_0(p-\eta)(2k_1+\sigma_3-2),
    \\
        A_7 &:=a_0(p-\eta)(2k_2+\sigma_4-2), & 
             &\qquad & 
        A_8 &:=b_0\chi_0(p-\eta)(2k_2+\sigma_4-2),
    \\ 
        A_9 &:=c_0\xi_0(p-\eta)(2k_2+\sigma_4-2), &
             &\qquad & 
        &
    \\
        A_{10} &:=d_1(p-\eta)(2k_1+\sigma_3-2), & 
             &\qquad & 
        A_{11} &:=d_1(2k_1+\sigma_3-2)(2k_1+\sigma_3-1),
    \\ 
        A_{12} &:=d_1(2k_1+\sigma_3-2)(2k_2+\sigma_4-2), & 
             &\qquad & 
        A_{13} &:=\alpha(2k_1+\sigma_3-2),
    \\ 
        A_{14} &:=\beta(2k_1+\sigma_3-2), &
             &\qquad & 
        &
    \\
        A_{15} &:=d_2(p-\eta)(2k_2+\sigma_4-2), & 
             &\qquad & 
        A_{16} &:=d_2(2k_1+\sigma_3-2)(2k_2+\sigma_4-2),
    \\ 
        A_{17} &:=d_2(2k_2+\sigma_4-2)(2k_2+\sigma_4-1), & 
             &\qquad & 
        A_{18} &:=\gamma(2k_2+\sigma_4-2),
    \\ 
        A_{19} &:=\delta(2k_2+\sigma_4-2). &
             &\qquad & 
        &
\end{alignat*}
%
Similarly, we can derive an estimate for 
$\frac{d}{dt}\int_\Omega \frac{(u+1)^{p-\eta}}{v^{2k_1+\sigma_1-2}}$, that is, 
%
\begin{align}\label{DI2-6}
    &\frac{d}{dt}\int_\Omega 
      \frac{(u+1)^{p-\eta}}{v^{2k_1+\sigma_1-2}}
     \notag\\[1.05mm]
    &\quad\,\le 
      -a_0(p-\eta-1)(p-\eta)\int_\Omega 
        \frac{(u+1)^{p+m-\eta-3}}
                {v^{2k_1+\sigma_1-2}}|\nabla u|^2\notag\\[1.05mm]
    &\quad\,\quad\,
      +b_0\chi_0(p-\eta-1)(p-\eta)\int_\Omega 
        \frac{(u+1)^{p+q-\eta-3}}
                {v^{3k_1+\sigma_1-2}}|\nabla u||\nabla v|
      \notag\\[1.05mm]
    &\quad\,\quad\,
      +c_0\xi_0(p-\eta-1)(p-\eta)\int_\Omega 
        \frac{(u+1)^{p+r-\eta-3}}
                {v^{2k_1+\sigma_1-2}w^{k_2}}|\nabla u||\nabla w|
      \notag\\[1.05mm]
    &\quad\,\quad\,
      +a_0(p-\eta)(2k_1+\sigma_1-2)\int_\Omega 
        \frac{(u+1)^{p+m-\eta-2}}
                {v^{2k_1+\sigma_1-1}}|\nabla u||\nabla v|
      \notag\\[1.05mm]
    &\quad\,\quad\,
      -(p-\eta)(2k_1+\sigma_1-2)\int_\Omega 
        \frac{(u+1)^{p-\eta-1}}
                {v^{2k_1+\sigma_1-1}}G(u)\chi(v)|\nabla v|^2
      \notag\\[1.05mm]
    &\quad\,\quad\,
      +c_0\xi_0(p-\eta)(2k_1+\sigma_1-2)\int_\Omega 
        \frac{(u+1)^{p+r-\eta-2}}
                {v^{2k_1+\sigma_1-1}w^{k_2}}|\nabla v||\nabla w|
      \notag\\[1.05mm]
    &\quad\,\quad\,
      +d_1(p-\eta)(2k_1+\sigma_1-2)\int_\Omega 
      \frac{(u+1)^{p-\eta-1}}{v^{2k_1+\sigma_1-1}}
      |\nabla u||\nabla v|
      \notag\\[1.05mm]
    &\quad\,\quad\,
      -d_1(2k_1+\sigma_1-2)(2k_1+\sigma_1-1)\int_\Omega 
      \frac{(u+1)^{p-\eta}}{v^{2k_1+\sigma_1}}
      |\nabla v|^2
      \notag\\[1.05mm]
    &\quad\,\quad\,
      -\alpha(2k_1+\sigma_1-2)\int_\Omega 
        \frac{u(u+1)^{p-\eta}}{v^{2k_1+\sigma_1-1}}
      \notag\\[1.05mm]
    &\quad\,\quad\,
      +\beta(2k_1+\sigma_1-2)\int_\Omega
        \frac{(u+1)^{p-\eta}}{v^{2k_1+\sigma_1-2}}.
\end{align}
%
We next estimate
$\frac{d}{dt}\int_\Omega \frac{(u+1)^{p-\eta}}{w^{2k_2+\sigma_2-2}}$. 
Using the first and third equations 
in \eqref{P}, we see that
%
\begin{align}\label{DI2-7}
    &\frac{d}{dt}\int_\Omega 
      \frac{(u+1)^{p-\eta}}{w^{2k_2+\sigma_2-2}}\notag\\[1.05mm]
    &\quad\, 
      =(p-\eta)\int_\Omega 
        \frac{(u+1)^{p-\eta-1}}{w^{2k_2+\sigma_2-2}}u_t
        -(2k_2+\sigma_2-2)\int_\Omega 
        \frac{(u+1)^{p-\eta}}{w^{2k_2+\sigma_2-1}}w_t
      \notag\\[1.05mm]
    &\quad\, 
      =(p-\eta)\int_\Omega 
        \frac{(u+1)^{p-\eta-1}}{w^{2k_2+\sigma_2-2}}
        \nabla \cdot [D(u)\nabla u-G(u)\chi(v)\nabla v+H(u)\xi(w)\nabla w]
        \notag\\[1.05mm]
    &\qquad\,\,\, 
        -(2k_2+\sigma_2-2)\int_\Omega 
        \frac{(u+1)^{p-\eta}}{w^{2k_2+\sigma_2-1}}
        (d_2\Delta w+\gamma u-\delta w)
      \notag\\[1.05mm]
    &\quad\,
       =:(p-\eta)J_4+(2k_2+\sigma_2-2)J_5.
\end{align}
%
Estimating $J_4, J_5$ in the same way as $J_1, J_2$, respectively, 
we obtain
%
\begin{align*}
    J_4
    &\le 
      -a_0(p-\eta-1)\int_\Omega 
        \frac{(u+1)^{p+m-\eta-3}}
                {w^{2k_2+\sigma_2-2}}|\nabla u|^2\notag\\
    &\quad\,
      +b_0\chi_0(p-\eta-1)\int_\Omega 
        \frac{(u+1)^{p+q-\eta-3}}
                {v^{k_1}w^{2k_2+\sigma_2-2}}|\nabla u||\nabla v|
      \notag\\
    &\quad\,
      +c_0\xi_0(p-\eta-1)\int_\Omega 
        \frac{(u+1)^{p+r-\eta-3}}
                {w^{3k_2+\sigma_2-2}}|\nabla u||\nabla w|
      \notag\\
    &\quad\,
      +a_0(2k_2+\sigma_2-2)\int_\Omega 
        \frac{(u+1)^{p+m-\eta-2}}
                {w^{2k_2+\sigma_2-1}}|\nabla u||\nabla w|
      \notag\\
    &\quad\,
      +b_0\chi_0(2k_2+\sigma_2-2)\int_\Omega 
        \frac{(u+1)^{p+q-\eta-2}}
                {v^{k_1}w^{2k_2+\sigma_2-1}}|\nabla v||\nabla w|
      \notag\\
    &\quad\,
      +c_0\xi_0(2k_2+\sigma_2-2)\int_\Omega 
        \frac{(u+1)^{p+r-\eta-2}}
                {w^{3k_2+\sigma_2-1}}|\nabla w|^2
\end{align*}
%
and 
%
\begin{align*}
    J_5
    &\le 
      d_2(p-\eta)\int_\Omega 
      \frac{(u+1)^{p-\eta-1}}{w^{2k_2+\sigma_2-1}}
      |\nabla u||\nabla w|
      -d_2(2k_2+\sigma_2-1)\int_\Omega 
      \frac{(u+1)^{p-\eta}}{w^{2k_2+\sigma_2}}
      |\nabla w|^2
      \notag\\
    &\quad\,
      -\gamma\int_\Omega 
        \frac{u(u+1)^{p-\eta}}{w^{2k_2+\sigma_2-1}}
      +\delta\int_\Omega
        \frac{(u+1)^{p-\eta}}{w^{2k_2+\sigma_2-2}}.
\end{align*}
%
Thus a combination of \eqref{DI2-7} and the above estimates for $J_4, J_5$ 
yields that
%
\begin{align}\label{DI2-10}
    \frac{d}{dt}\int_\Omega 
      \frac{(u+1)^{p-\eta}}{w^{2k_2+\sigma_2-2}}
    &\le 
      -a_0(p-\eta-1)(p-\eta)\int_\Omega 
        \frac{(u+1)^{p+m-\eta-3}}
                {w^{2k_2+\sigma_2-2}}|\nabla u|^2\notag\\
    &\quad\,
      +b_0\chi_0(p-\eta-1)(p-\eta)\int_\Omega 
        \frac{(u+1)^{p+q-\eta-3}}
                {v^{k_1}w^{2k_2+\sigma_2-2}}|\nabla u||\nabla v|
      \notag\\
    &\quad\,
      +c_0\xi_0(p-\eta-1)(p-\eta)\int_\Omega 
        \frac{(u+1)^{p+r-\eta-3}}
                {w^{3k_2+\sigma_2-2}}|\nabla u||\nabla w|
      \notag\\
    &\quad\,
      +a_0(p-\eta)(2k_2+\sigma_2-2)\int_\Omega 
        \frac{(u+1)^{p+m-\eta-2}}
                {w^{2k_2+\sigma_2-1}}|\nabla u||\nabla w|
      \notag\\
    &\quad\,
      +b_0\chi_0(p-\eta)(2k_2+\sigma_2-2)\int_\Omega 
        \frac{(u+1)^{p+q-\eta-2}}
                {v^{k_1}w^{2k_2+\sigma_2-1}}|\nabla v||\nabla w|
      \notag\\
    &\quad\,
      +c_0\xi_0(p-\eta)(2k_2+\sigma_2-2)\int_\Omega 
        \frac{(u+1)^{p+r-\eta-2}}
                {w^{3k_2+\sigma_2-1}}|\nabla w|^2
      \notag\\
    &\quad\,
      +d_2(p-\eta)(2k_2+\sigma_2-2)\int_\Omega 
      \frac{(u+1)^{p-\eta-1}}{w^{2k_2+\sigma_2-1}}
      |\nabla u||\nabla w|
      \notag\\
    &\quad\,
      -d_2(2k_2+\sigma_2-2)(2k_2+\sigma_2-1)\int_\Omega 
      \frac{(u+1)^{p-\eta}}{w^{2k_2+\sigma_2}}
      |\nabla w|^2
      \notag\\
    &\quad\,
      -\gamma(2k_2+\sigma_2-2)\int_\Omega 
        \frac{u(u+1)^{p-\eta}}{w^{2k_2+\sigma_2-1}}
      \notag\\
    &\quad\,
      +\delta(2k_2+\sigma_2-2)\int_\Omega
        \frac{(u+1)^{p-\eta}}{w^{2k_2+\sigma_2-2}}.
\end{align}
%
Adding estimates \eqref{DI2-6} and \eqref{DI2-10}, 
and moving 
$\int_\Omega \frac{(u+1)^{p-\eta}}{v^{2k_1+\sigma_1}}
      |\nabla v|^2$ and 
$\int_\Omega 
      \frac{(u+1)^{p-\eta}}{w^{2k_2+\sigma_2}}
      |\nabla w|^2$
to the left-hand side, we have
%
\begin{align}\label{DI2-11}
    &\frac{d}{dt}\int_\Omega 
      \frac{(u+1)^{p-\eta}}{v^{2k_1+\sigma_1-2}}
    +\frac{d}{dt}\int_\Omega 
      \frac{(u+1)^{p-\eta}}{w^{2k_2+\sigma_2-2}}
    +B_{19}\int_\Omega 
      \frac{(u+1)^{p-\eta}}{v^{2k_1+\sigma_1}}
      |\nabla v|^2
    +B_{20}\int_\Omega 
      \frac{(u+1)^{p-\eta}}{w^{2k_2+\sigma_2}}
      |\nabla w|^2
    \notag\\
    &\quad\,\le 
      -B_1\int_\Omega 
        \frac{(u+1)^{p+m-\eta-3}}
                {v^{2k_1+\sigma_1-2}}|\nabla u|^2
      -B_2\int_\Omega 
        \frac{(u+1)^{p+m-\eta-3}}
                {w^{2k_2+\sigma_2-2}}|\nabla u|^2\notag\\
    &\qquad\,\,
      +B_3\int_\Omega 
        \frac{(u+1)^{p+m-\eta-2}}
                {v^{2k_1+\sigma_1-1}}|\nabla u||\nabla v|
      +B_4\int_\Omega 
        \frac{(u+1)^{p+m-\eta-2}}
                {w^{2k_2+\sigma_2-1}}|\nabla u||\nabla w|
      \notag\\
    &\qquad\,\,
      +B_5\int_\Omega 
        \frac{(u+1)^{p+q-\eta-3}}
                {v^{3k_1+\sigma_1-2}}|\nabla u||\nabla v|
      -B_6\int_\Omega 
        \frac{(u+1)^{p-\eta-1}}
                {v^{2k_1+\sigma_1-1}}G(u)\chi(v)|\nabla v|^2
      \notag\\
    &\qquad\,\,
      +B_7\int_\Omega 
        \frac{(u+1)^{p+q-\eta-3}}
                {v^{k_1}w^{2k_2+\sigma_2-2}}|\nabla u||\nabla v|
      +B_8\int_\Omega 
        \frac{(u+1)^{p+r-\eta-3}}
                {w^{3k_2+\sigma_2-2}}|\nabla u||\nabla w|
      \notag\\
    &\qquad\,\,
      +B_9\int_\Omega 
        \frac{(u+1)^{p+r-\eta-3}}
                {v^{2k_1+\sigma_1-2}w^{k_2}}|\nabla u||\nabla w|
      +B_{10}\int_\Omega 
        \frac{(u+1)^{p+r-\eta-2}}
                {v^{2k_1+\sigma_1-1}w^{k_2}}|\nabla v||\nabla w|
      \notag\\
    &\qquad\,\,
      +B_{11}\int_\Omega 
        \frac{(u+1)^{p+q-\eta-2}}
                {v^{k_1}w^{2k_2+\sigma_2-1}}|\nabla v||\nabla w|
      +B_{12}\int_\Omega 
        \frac{(u+1)^{p+r-\eta-2}}
                {w^{3k_2+\sigma_2-1}}|\nabla w|^2
      \notag\\
    &\qquad\,\,
      +B_{13}\int_\Omega 
      \frac{(u+1)^{p-\eta-1}}{v^{2k_1+\sigma_1-1}}
      |\nabla u||\nabla v|
      +B_{14}\int_\Omega 
      \frac{(u+1)^{p-\eta-1}}{w^{2k_2+\sigma_2-1}}
      |\nabla u||\nabla w|
      \notag\\
    &\qquad\,\,
      -B_{15}\int_\Omega 
        \frac{u(u+1)^{p-\eta}}{v^{2k_1+\sigma_1-1}}
      +B_{16}\int_\Omega
        \frac{(u+1)^{p-\eta}}{v^{2k_1+\sigma_1-2}}
      \notag\\
    &\qquad\,\,
      -B_{17}\int_\Omega 
        \frac{u(u+1)^{p-\eta}}{w^{2k_2+\sigma_2-1}}
      +B_{18}\int_\Omega
        \frac{(u+1)^{p-\eta}}{w^{2k_2+\sigma_2-2}},
\end{align}
%
where
%
\begin{alignat*}{3}
        B_1 &:=a_0(p-\eta-1)(p-\eta), & 
             &\qquad & 
        B_2 &:=a_0(p-\eta-1)(p-\eta),
    \\ 
        B_3 &:=a_0(p-\eta)(2k_1+\sigma_1-2), &
             &\qquad & 
        B_4 &:=a_0(p-\eta)(2k_2+\sigma_2-2),
    \\
        B_5 &:=b_0\chi_0(p-\eta-1)(p-\eta), &
             &\qquad &
        B_6 &:=(p-\eta)(2k_1+\sigma_1-2),
    \\
        B_7 &:=b_0\chi_0(p-\eta-1)(p-\eta),& 
             &\qquad & 
        B_8 &:=c_0\xi_0(p-\eta-1)(p-\eta),
    \\ 
        B_9 &:=c_0\xi_0(p-\eta-1)(p-\eta), &
             &\qquad & 
        B_{10} &:=c_0\xi_0(p-\eta)(2k_1+\sigma_1-2),
    \\
        B_{11} &:=b_0\chi_0(p-\eta)(2k_2+\sigma_2-2),& 
             &\qquad & 
        B_{12} &:=c_0\xi_0(p-\eta)(2k_2+\sigma_2-2),
    \\ 
        B_{13} &:=d_1(p-\eta)(2k_1+\sigma_1-2), & 
             &\qquad & 
        B_{14} &:=d_2(p-\eta)(2k_2+\sigma_2-2),
    \\ 
        B_{15} &:=\alpha(2k_1+\sigma_1-2), &
             &\qquad & 
        B_{16} &:=\beta(2k_1+\sigma_1-2),
    \\ 
        B_{17} &:=\gamma(2k_2+\sigma_2-2), &
             &\qquad & 
        B_{18} &:=\delta(2k_2+\sigma_2-2),
    \\ 
        B_{19} &:=d_1(2k_1+\sigma_1-2)(2k_1+\sigma_1-1), & 
             &\qquad & 
        B_{20} &:=d_2(2k_2+\sigma_2-2)(2k_2+\sigma_2-1).
\end{alignat*}
%
Adding \eqref{DI2-11} and \eqref{DI2-5} multiplied by $\ep_{03}>0$ 
and dropping the nine terms containing 
$A_i, B_j$ ($i \in \{1, 5, 13, 18\}$, $j \in \{1, 2, 6, 15, 17\}$), 
we can see that the following inequality holds:
%
\begin{align}\label{DI2-12}
    &\quad\,\ep_{03}\frac{d}{dt}\int_\Omega 
      \frac{(u+1)^{p-\eta}}
             {v^{2k_1+\sigma_3-2}w^{2k_2+\sigma_4-2}}
    +\frac{d}{dt}\int_\Omega 
      \frac{(u+1)^{p-\eta}}{v^{2k_1+\sigma_1-2}}
    +\frac{d}{dt}\int_\Omega 
      \frac{(u+1)^{p-\eta}}{w^{2k_2+\sigma_2-2}}\notag\\[0.5mm]
    &\quad
    +B_{19}\int_\Omega 
      \frac{(u+1)^{p-\eta}}{v^{2k_1+\sigma_1}}
      |\nabla v|^2
    +B_{20}\int_\Omega 
      \frac{(u+1)^{p-\eta}}{w^{2k_2+\sigma_2}}
      |\nabla w|^2
    \notag\\[0.5mm]
    &\quad
    +\ep_{03}A_{11}\int_\Omega 
      \frac{(u+1)^{p-\eta}}{v^{2k_1+\sigma_3}w^{2k_2+\sigma_4-2}}|\nabla v|^2
    +\ep_{03}A_{17}\int_\Omega 
      \frac{(u+1)^{p-\eta}}{v^{2k_1+\sigma_3-2}w^{2k_2+\sigma_4}}
      |\nabla w|^2
    \notag\\[0.5mm]
    &\le 
      \underset{\hfill\bm{=:\,I_1}}{\underline{\ep_{03}A_4\int_\Omega 
        \frac{(u+1)^{p+m-\eta-2}}
                {v^{2k_1+\sigma_3-1}w^{2k_2+\sigma_4-2}}|\nabla u||\nabla v|}}
      +\underset{\hfill\bm{=:\,I_2}}{\underline{B_3\int_\Omega 
        \frac{(u+1)^{p+m-\eta-2}}
                {v^{2k_1+\sigma_1-1}}|\nabla u||\nabla v|}}
      \notag\\[-1.1mm]
    &\quad
      +\underset{\hfill\bm{=:\,I_3}}{\underline{\ep_{03}A_7\int_\Omega 
        \frac{(u+1)^{p+m-\eta-2}}
                {v^{2k_1+\sigma_3-2}w^{2k_2+\sigma_4-1}}|\nabla u||\nabla w|}}
      +\underset{\hfill\bm{=:\,I_4}}{\underline{B_4\int_\Omega 
        \frac{(u+1)^{p+m-\eta-2}}
                {w^{2k_2+\sigma_2-1}}|\nabla u||\nabla w|}}
      \notag\\[-1.1mm]
    &\quad
      +\underset{\hfill\bm{=:\,I_5}}{\underline{\ep_{03}A_2\int_\Omega 
        \frac{(u+1)^{p+q-\eta-3}}
                {v^{3k_1+\sigma_3-2}w^{2k_2+\sigma_4-2}}|\nabla u||\nabla v|}}
      +\underset{\hfill\bm{=:\,I_6}}{\underline{B_5\int_\Omega 
        \frac{(u+1)^{p+q-\eta-3}}
                {v^{3k_1+\sigma_1-2}}|\nabla u||\nabla v|}}
      \notag\\[-1.1mm]
    &\quad
      +\underset{\hfill\bm{=:\,I_7}}{\underline{B_7\int_\Omega 
        \frac{(u+1)^{p+q-\eta-3}}
                {v^{k_1}w^{2k_2+\sigma_2-2}}|\nabla u||\nabla v|}}
      +\underset{\hfill\bm{=:\,I_8}}{\underline{B_8\int_\Omega 
        \frac{(u+1)^{p+r-\eta-3}}
                {w^{3k_2+\sigma_2-2}}|\nabla u||\nabla w|}}
      \notag\\[-1.1mm]
    &\quad
      +\underset{\hfill\bm{=:\,I_9}}{\underline{\ep_{03}A_3\int_\Omega 
        \frac{(u+1)^{p+r-\eta-3}}
                {v^{2k_1+\sigma_3-2}w^{3k_2+\sigma_4-2}}|\nabla u||\nabla w|}}
      +\underset{\hfill\bm{=:\,I_{10}}}{\underline{B_9\int_\Omega 
        \frac{(u+1)^{p+r-\eta-3}}
                {v^{2k_1+\sigma_1-2}w^{k_2}}|\nabla u||\nabla w|}}
      \notag\\[-1.1mm]
    &\quad
      +\underset{\hfill\bm{=:\,I_{11}}}{\underline{\ep_{03}A_6\int_\Omega 
        \frac{(u+1)^{p+r-\eta-2}}
                {v^{2k_1+\sigma_3-1}w^{3k_2+\sigma_4-2}}|\nabla v||\nabla w|}}
      +\underset{\hfill\bm{=:\,I_{12}}}{\underline{\ep_{03}A_8\int_\Omega 
        \frac{(u+1)^{p+q-\eta-2}}
                {v^{3k_1+\sigma_3-2}w^{2k_2+\sigma_4-1}}|\nabla v||\nabla w|}}
    \notag\\[-1.1mm]
    &\quad
      +\underset{\hfill\bm{=:\,I_{13}}}{\underline{B_{10}\int_\Omega 
        \frac{(u+1)^{p+r-\eta-2}}
                {v^{2k_1+\sigma_1-1}w^{k_2}}|\nabla v||\nabla w|}}
      +\underset{\hfill\bm{=:\,I_{14}}}{\underline{B_{11}\int_\Omega 
        \frac{(u+1)^{p+q-\eta-2}}
                {v^{k_1}w^{2k_2+\sigma_2-1}}|\nabla v||\nabla w|}}
      \notag\\[-1.1mm]
    &\quad  
    +\underset{\hfill\bm{=:\,I_{15}}}{\underline{\ep_{03}A_9\int_\Omega 
        \frac{(u+1)^{p+r-\eta-2}}
                {v^{2k_1+\sigma_3-2}w^{3k_2+\sigma_4-1}}|\nabla w|^2}}
    +\underset{\hfill\bm{=:\,I_{16}}}{\underline{B_{12}\int_\Omega 
        \frac{(u+1)^{p+r-\eta-2}}
                {w^{3k_2+\sigma_2-1}}|\nabla w|^2}}
      \notag\\[-1.1mm]
    &\quad
      +\underset{\hfill\bm{=:\,I_{17}}}{\underline{\ep_{03}A_{10}\int_\Omega 
      \frac{(u+1)^{p-\eta-1}}{v^{2k_1+\sigma_3-1}w^{2k_2+\sigma_4-2}}
      |\nabla u||\nabla v|}}
      +\underset{\hfill\bm{=:\,I_{18}}}{\underline{B_{13}\int_\Omega 
      \frac{(u+1)^{p-\eta-1}}{v^{2k_1+\sigma_1-1}}
      |\nabla u||\nabla v|}}
      \notag\\[-1.1mm]
    &\quad
      +\underset{\hfill\bm{=:\,I_{19}}}{\underline{\ep_{03}A_{15}\int_\Omega 
      \frac{(u+1)^{p-\eta-1}}{v^{2k_1+\sigma_3-2}w^{2k_2+\sigma_4-1}}
      |\nabla u||\nabla w|}}
      +\underset{\hfill\bm{=:\,I_{20}}}{\underline{B_{14}\int_\Omega 
      \frac{(u+1)^{p-\eta-1}}{w^{2k_2+\sigma_2-1}}
      |\nabla u||\nabla w|}}
      \notag\\[-1.1mm]
    &\quad
      +\underset{\hfill\bm{=:\,I_{21}}}{\underline{\ep_{03}(A_{12}+A_{16})\int_\Omega 
      \frac{(u+1)^{p-\eta}}{v^{2k_1+\sigma_3-1}w^{2k_2+\sigma_4-1}}
      |\nabla v||\nabla w|}}
    \notag\\[-1.1mm]
    &\quad
      +\underset{\hfill\bm{=:\,I_{22}}}{\underline{\ep_{03}A_{14}\int_\Omega
        \frac{(u+1)^{p-\eta}}{v^{2k_1+\sigma_3-2}w^{2k_2+\sigma_4-2}}}}
      +\underset{\hfill\bm{=:\,I_{23}}}{\underline{B_{16}\int_\Omega
        \frac{(u+1)^{p-\eta}}{v^{2k_1+\sigma_1-2}}}}
      \notag\\[-0.5mm]
    &\quad
      +\underset{\hfill\bm{=:\,I_{24}}}{\underline{\ep_{03}A_{19}\int_\Omega
        \frac{(u+1)^{p-\eta}}{v^{2k_1+\sigma_3-2}w^{2k_2+\sigma_4-2}}}}
      +\underset{\hfill\bm{=:\,I_{25}}}{\underline{B_{18}\int_\Omega
        \frac{(u+1)^{p-\eta}}{w^{2k_2+\sigma_2-2}}}}.
\end{align}
%
We now estimate the twenty-five terms 
$I_1$--$I_{25}$ by dividing it into the four steps. 
Here we note from \eqref{infv}, \eqref{infw} and the condition 
$q, r<2$ that $I_5, I_6, I_8, I_9$ can be estimated by 
$I_{17}, I_{18}, I_{20}, I_{19}$, 
and that $I_{11}, I_{12}$ can be estimated by $I_{21}$, respectively. 

{\bf Step 1.} We estimate the fourteen terms containing $|\nabla u|$ 
(i.e., $I_1$--$I_{10}$, $I_{17}$--$I_{20}$) so that the integral 
$\int_\Omega (u+1)^{p-2}|\nabla u|^2$ appears. 
For instance, as to an estimate for $I_{1}$, 
for all $\ep_1>0$, 
it can be obtained upon Lemma~\ref{Young} that
%
\begin{align}\label{preest1}
      I_1
      &=\ep_{03}A_4\int_\Omega (u+1)^{\frac{p-2}{2}}|\nabla u| \cdot
          \frac{(u+1)^{\frac{p+2m-2\eta-2}{2}}}{v^{2k_1+\sigma_3-1}w^{2k_2+\sigma_4-2}}
          |\nabla v|
      \notag\\
      &\le 
        \ep_{03}A_4\cdot\frac{\ep_1}{2}\int_\Omega (u+1)^{p-2}|\nabla u|^2
        +\ep_{03}A_4\cdot\frac{1}{2\ep_1}
          \int_\Omega\frac{(u+1)^{p+2m-2\eta-2}}
                                   {v^{4k_1+2\sigma_3-2}w^{4k_2+2\sigma_4-4}}
                            |\nabla v|^2
        \notag\\
      &\le\frac{\ep_{03}A_4}{2}\cdot\ep_1\int_\Omega (u+1)^{p-2}|\nabla u|^2
        +\frac{\ep_{03}A_4'}{2}\frac{1}{\ep_1}\cdot\ep_1^{\theta_1}
          \int_\Omega \frac{(u+1)^{p-\eta}}
                             {v^{(4k_1+2\sigma_3-2)\theta_1}w^{(4k_2+2\sigma_4-4)\theta_1}}
                            |\nabla v|^2
        \notag\\
      &\quad\ \,
        +\frac{\ep_{03}A_4''}{2}\frac{1}{\ep_1}\cdot\ep_1^{-\frac{\theta_1}{\theta_1-1}}
          \int_\Omega |\nabla v|^2
\end{align}
%
with $\theta_1:=\frac{p-\eta}{p+2m-2\eta-2}>1$ by the condition 
$\eta>2(m-1)$ (see \eqref{condieta2}). 
Here we can check that 
$(4k_1+2\sigma_3-2)\theta_1 \ge 2k_1+\sigma_3$ and 
$(4k_2+2\sigma_4-4)\theta_1 \ge 2k_2+\sigma_4-2$. 
Indeed, as to the former, 
a simple calculation and the assumption $\sigma_3>2(1-k_1)$ 
(see Lemma \ref{difineq2}) as well as \eqref{condip2} yield that
%
\begin{align*}
    &(4k_1+2\sigma_3-2)\theta_1-(2k_1+\sigma_3)
    \notag\\
    &\quad\,=
      (2k_1+\sigma_3-1)\Big(2\theta_1-1-\frac{1}{2k_1+\sigma_3-1}\Big)
    \notag\\
    &\quad\,=
      \frac{2k_1+\sigma_3-1}{p+2m-2\eta-2}
      \Big(2p-2\eta-(p+2m-2\eta-2)-\frac{p+2m-2\eta-2}{2k_1+\sigma_3-1}\Big)
    \notag\\
    &\quad\,=
      \frac{2k_1+\sigma_3-1}{p+2m-2\eta-2}
      \Big(\frac{2k_1+\sigma_3-2}{2k_1+\sigma_3-1}p-2(m-1)-\frac{2(m-\eta-1)}{2k_1+\sigma_3-1}\Big)
    \notag\\
    &\quad\,=
      \frac{2k_1+\sigma_3-2}{p+2m-2\eta-2}
      \Big(p-\frac{2[(m-1)(2k_1+\sigma_3-1)+(m-\eta-1)]}{2k_1+\sigma_3-2}\Big)
    \notag\\
    &\quad\,>0.
\end{align*}
%
As to the later, the fact $\theta_1>1$ and the assumption 
$\sigma_4>2(1-k_2)$ derive that
%
\begin{align*}
    (4k_2+2\sigma_4-4)\theta_1-(2k_2+\sigma_4-2)
    =(2k_2+\sigma_4-2)(2\theta_1-1)>0.
\end{align*}
%
Thus we see from \eqref{infv}, \eqref{infw} that
%
\begin{align*}
    &v^{(4k_1+2\sigma_3-2)\theta_1}
    \ge \mu_1^{(4k_1+2\sigma_3-2)\theta_1-(2k_1+\sigma_3)}
          v^{2k_1+\sigma_3},\notag\\
    &w^{(4k_2+2\sigma_4-4)\theta_1}
    \ge \mu_2^{(4k_2+2\sigma_4-4)\theta_1-(2k_2+\sigma_4-2)}
         w^{2k_2+\sigma_4-2},
\end{align*}
%
which together with \eqref{preest1} imply that
%
\begin{align}\label{estI1}
      I_1
      &\le\frac{\ep_{03}A_4}{2}\cdot\ep_1\int_\Omega (u+1)^{p-2}|\nabla u|^2
        +\frac{\ep_{03}A_4'}{2}\cdot\ep_1^{\theta_1-1}
          \int_\Omega \frac{(u+1)^{p-\eta}}
                             {v^{2k_1+\sigma_3}w^{2k_2+\sigma_4-2}}
                            |\nabla v|^2
        \notag\\
      &\quad\ \,
        +\frac{\ep_{03}A_4''}{2}\cdot\ep_1^{-\frac{\theta_1}{\theta_1-1}-1}
          \int_\Omega |\nabla v|^2.
\end{align}
%
Proceeding similarly as the above estimate, we also obtain that 
for all $\ep_i>0$ there exist constants $\theta_i>1$ 
($i \in \{2, 3, 4, 7, 10, 17, 18, 19, 20\}$) such that
%
\begin{align}
      I_2
      &\le\frac{B_3}{2}\cdot\ep_2\int_\Omega (u+1)^{p-2}|\nabla u|^2
        +\frac{B_3'}{2}\cdot\ep_2^{\theta_2-1}
          \int_\Omega \frac{(u+1)^{p-\eta}}
                             {v^{2k_1+\sigma_1}}
                            |\nabla v|^2
        \notag\\
      &\quad\ \,
        +\frac{B_3''}{2}\ep_2^{-\frac{\theta_2}{\theta_2-1}-1}
          \int_\Omega |\nabla v|^2,
      \\
      I_3
      &\le\frac{\ep_{03}A_7}{2}\cdot\ep_3\int_\Omega (u+1)^{p-2}|\nabla u|^2
        +\frac{\ep_{03}A_7'}{2}\cdot\ep_3^{\theta_3-1}
          \int_\Omega \frac{(u+1)^{p-\eta}}
                             {v^{2k_1+\sigma_3-2}w^{2k_2+\sigma_4}}
                            |\nabla w|^2
        \notag\\
      &\quad\ \,
        +\frac{\ep_{03}A_7''}{2}\cdot\ep_3^{-\frac{\theta_3}{\theta_3-1}-1}
          \int_\Omega |\nabla w|^2,
      \\
      I_4
      &\le\frac{B_4}{2}\cdot\ep_4\int_\Omega (u+1)^{p-2}|\nabla u|^2
        +\frac{B_4'}{2}\cdot\ep_4^{\theta_4-1}
          \int_\Omega \frac{(u+1)^{p-\eta}}
                             {w^{2k_2+\sigma_2}}
                            |\nabla w|^2
        \notag\\
      &\quad\ \,
        +\frac{B_4''}{2}\ep_4^{-\frac{\theta_4}{\theta_4-1}-1}
          \int_\Omega |\nabla w|^2,
      \\
      I_7
      &\le\frac{B_7}{2}\cdot\ep_7\int_\Omega (u+1)^{p-2}|\nabla u|^2
        +\frac{B_7'}{2}\cdot\ep_7^{\theta_7-1}
          \int_\Omega \frac{(u+1)^{p-\eta}}
                             {v^{2k_1+\sigma_1}w^{2k_2+\sigma_2-2}}
                             |\nabla v|^2
        \notag\\
      &\quad\ \,
        +\frac{B_7''}{2}\ep_7^{-\frac{\theta_7}{\theta_7-1}-1}
          \int_\Omega |\nabla v|^2,
      \\
      I_{10}
      &\le\frac{B_9}{2}\cdot\ep_{10}\int_\Omega (u+1)^{p-2}|\nabla u|^2
        +\frac{B_9'}{2}\cdot\ep_{10}^{\theta_{10}-1}
          \int_\Omega \frac{(u+1)^{p-\eta}}
                             {v^{2k_1+\sigma_1-2}w^{2k_2+\sigma_2}}
                             |\nabla w|^2
        \notag\\
      &\quad\ \,
        +\frac{B_9''}{2}\ep_{10}^{-\frac{\theta_{10}}{\theta_{10}-1}-1}
          \int_\Omega |\nabla w|^2,
      \\
      I_{17}
      &\le\frac{\ep_{03}A_{10}}{2}\cdot\ep_{17}\int_\Omega (u+1)^{p-2}|\nabla u|^2
        +\frac{\ep_{03}A_{10}'}{2}\cdot\ep_{17}^{\theta_{17}-1}
          \int_\Omega \frac{(u+1)^{p-\eta}}
                             {v^{2k_1+\sigma_3}w^{2k_2+\sigma_4-2}}
                            |\nabla v|^2
        \notag\\
      &\quad\ \,
        +\frac{\ep_{03}A_{10}''}{2}\cdot\ep_{17}^{-\frac{\theta_{17}}{\theta_{17}-1}-1}
          \int_\Omega |\nabla v|^2,
      \\
      I_{18}
      &\le\frac{B_{13}}{2}\cdot\ep_{18}\int_\Omega (u+1)^{p-2}|\nabla u|^2
        +\frac{B_{13}'}{2}\cdot\ep_{18}^{\theta_{18}-1}
          \int_\Omega \frac{(u+1)^{p-\eta}}
                             {v^{2k_1+\sigma_1}}
                            |\nabla v|^2
        \notag\\
      &\quad\ \,
        +\frac{B_{13}''}{2}\cdot\ep_{18}^{-\frac{\theta_{18}}{\theta_{18}-1}-1}
          \int_\Omega |\nabla v|^2,
      \\
      I_{19}
      &\le\frac{\ep_{03}A_{15}}{2}\cdot\ep_{19}\int_\Omega (u+1)^{p-2}|\nabla u|^2
        +\frac{\ep_{03}A_{15}'}{2}\cdot\ep_{19}^{\theta_{19}-1}
          \int_\Omega \frac{(u+1)^{p-\eta}}
                             {v^{2k_1+\sigma_3-2}w^{2k_2+\sigma_4}}
                            |\nabla w|^2
        \notag\\
      &\quad\ \,
        +\frac{\ep_{03}A_{15}''}{2}\cdot\ep_{19}^{-\frac{\theta_{19}}{\theta_{19}-1}-1}
          \int_\Omega |\nabla w|^2,
      \\
      I_{20}
      &\le\frac{B_{14}}{2}\cdot\ep_{20}\int_\Omega (u+1)^{p-2}|\nabla u|^2
        +\frac{B_{14}'}{2}\cdot\ep_{20}^{\theta_{20}-1}
          \int_\Omega \frac{(u+1)^{p-\eta}}
                             {w^{2k_2+\sigma_2}}
                            |\nabla w|^2
        \notag\\
      &\quad\ \,
        +\frac{B_{14}''}{2}\cdot\ep_{20}^{-\frac{\theta_{20}}{\theta_{20}-1}-1}
          \int_\Omega |\nabla w|^2.
\end{align}
%

{\bf Step 2.} We estimate the five terms containing $|\nabla v||\nabla w|$ 
(i.e., $I_{11}$--$I_{14}$, $I_{21}$). 
Here we can omit estimates for $I_{11}, I_{12}$ as mentioned above. 
As to an estimate for $I_{13}$, we see that for all $\ep_{13}>0$, 
%
\begin{align}
      I_{13}
      &=B_{10}\int_\Omega 
          \frac{(u+1)^{\frac{p+2r-\eta-4}{2}}}
                 {v^{\frac{2k_1+2\sigma_1-\sigma_3}{2}}w^{-\frac{\sigma_4}{2}}}
                 |\nabla v| \cdot
          \frac{(u+1)^{\frac{p-\eta}{2}}}
                 {v^{\frac{2k_1+\sigma_3-2}{2}}w^{\frac{2k_2+\sigma_4}{2}}}
                 |\nabla w|
      \notag\\
      &\le 
        B_{10}\cdot\frac{1}{2\ep_{13}}\int_\Omega 
        \frac{(u+1)^{p+2r-\eta-4}}{v^{2k_1+2\sigma_1-\sigma_3}w^{-\sigma_4}}
        |\nabla v|^2
        +B_{10} \cdot \frac{\ep_{13}}{2}\int_\Omega
        \frac{(u+1)^{p-\eta}}{v^{2k_1+\sigma_3-2}w^{2k_2+\sigma_4}}
        |\nabla w|^2
        \notag\\
      &\le\frac{B_{10}'}{2}\frac{1}{\ep_{13}}\cdot \ep_{13}^{\theta_{13}}
        \int_\Omega 
        \frac{(u+1)^{p-\eta}}{v^{(2k_1+2\sigma_1-\sigma_3)\theta_{13}}
        w^{-\sigma_4\theta_{13}}}
        |\nabla v|^2
        +\frac{B_{10}''}{2}\frac{1}{\ep_{13}}\cdot \ep_{13}^{-\frac{\theta_{13}}{\theta_{13}-1}}
        \int_\Omega |\nabla v|^2
        \notag\\
      &\quad\ \,
        +\frac{B_{10}}{2} \cdot \ep_{13}\int_\Omega
        \frac{(u+1)^{p-\eta}}{v^{2k_1+\sigma_3-2}w^{2k_2+\sigma_4}}
        |\nabla w|^2
      \notag\\
      &\le\frac{B_{10}'}{2}\cdot \ep_{13}^{\theta_{13}-1}
        \int_\Omega 
        \frac{(u+1)^{p-\eta}}{v^{2k_1+\sigma_1}}
        |\nabla v|^2
        +\frac{B_{10}''}{2}\cdot \ep_{13}^{-\frac{\theta_{13}}{\theta_{13}-1}-1}
        \int_\Omega |\nabla v|^2
        \notag\\
      &\quad\ \,
        +\frac{B_{10}}{2} \cdot \ep_{13}\int_\Omega
        \frac{(u+1)^{p-\eta}}{v^{2k_1+\sigma_3-2}w^{2k_2+\sigma_4}}
        |\nabla w|^2
\end{align}
%
holds, where $\theta_{13}:=\frac{p-\eta}{p+2r-\eta-4}>1$ 
by the condition $r<2$. 
Here we used the facts that 
$(2k_1+2\sigma_1-\sigma_3)\theta_{13} \ge 2k_1+\sigma_1$ and 
$-\sigma_4\theta_{13}>0$ due to 
$\theta_{13}>1, \sigma_1>0, \sigma_3<0$ and $\sigma_4<0$, 
respectively. 
Similarly, we establish that for all $\ep_{14}>0$ there exists 
$\theta_{14}>1$ such that
%
\begin{align}
      I_{14}
      &\le\frac{B_{11}}{2}\cdot \ep_{14}\int_\Omega
        \frac{(u+1)^{p-\eta}}{v^{2k_1+\sigma_3}w^{2k_2+\sigma_4-2}}
        |\nabla v|^2
        \notag\\
      &\quad\ \,+\frac{B_{11}'}{2}\cdot \ep_{14}^{\theta_{14}-1}
        \int_\Omega \frac{(u+1)^{p-\eta}}{w^{2k_2+\sigma_2}}|\nabla w|^2
        +\frac{B_{11}''}{2} \cdot \ep_{14}^{-\frac{\theta_{14}}{\theta_{14}-1}-1}
        \int_\Omega
        |\nabla w|^2.
\end{align}
%
Also, we can derive that for all $\ep_{21}>0$, 
%
\begin{align}
      I_{21}
      &=\ep_{03}A_{12}'\int_\Omega
        \frac{(u+1)^{\frac{p-\eta}{2}}}{v^\frac{2k_1+\sigma_1}{2}}|\nabla v| \cdot
        \frac{(u+1)^{\frac{p-\eta}{2}}}{v^\frac{2k_1-\sigma_1+2\sigma_3-2}{2}
        w^{2k_2+\sigma_4-1}}|\nabla w|
        \notag\\
      &\le\frac{\ep_{03}A_{12}'}{2}\int_\Omega
        \frac{(u+1)^{p-\eta}}{v^{2k_1+\sigma_1}}|\nabla v|^2
        +\frac{\ep_{03}A_{12}'}{2}\int_\Omega
        \frac{(u+1)^{p-\eta}}{v^{2k_1-\sigma_1+2\sigma_3-2}
        w^{4k_2+2\sigma_4-2}}|\nabla w|^2
        \notag\\
      &\le\frac{\ep_{03}A_{12}'}{2}\int_\Omega
        \frac{(u+1)^{p-\eta}}{v^{2k_1+\sigma_1}}|\nabla v|^2
        +\frac{\ep_{03}A_{12}'}{2}\int_\Omega
        \frac{(u+1)^{p-\eta}}{w^{2k_2+\sigma_2}}|\nabla w|^2
\end{align}
%
holds, since $2k_1-\sigma_1+2\sigma_3-2 \ge 0$ and 
$4k_2+2\sigma_4-2 \ge 2k_2+\sigma_2$ due to $\sigma_3>2(1-k_1)$ 
and $\sigma_4>2(1-k_2)$, respectively.

{\bf Step 3.} We estimate the two terms containing $|\nabla w|^2$ 
(i.e., $I_{15}$, $I_{16}$). 
As to an estimate for $I_{15}$, we deduce that for all $\ep_{15}>0$,
%
\begin{align}
      I_{15}
      &\le \ep_{03}A_9 \cdot \ep_{15}^{\theta_{15}}\int_\Omega
      \frac{(u+1)^{p-\eta}}{v^{(2k_1+\sigma_3-2)\theta_{15}}
      w^{(3k_2+\sigma_4-1)\theta_{15}}}|\nabla w|^2
      +\ep_{03}A_9 \cdot \ep_{15}^{-\frac{\theta_{15}}{\theta_{15}-1}}
      \int_\Omega |\nabla w|^2
      \notag\\
      &\le \ep_{03}A_9 \cdot \ep_{15}^{\theta_{15}}\int_\Omega
      \frac{(u+1)^{p-\eta}}{v^{2k_1+\sigma_3-2}
      w^{2k_2+\sigma_4}}|\nabla w|^2
      +\ep_{03}A_9 \cdot \ep_{15}^{-\frac{\theta_{15}}{\theta_{15}-1}}
      \int_\Omega |\nabla w|^2
\end{align}
%
holds, where $\theta_{15}:=\frac{p-\eta}{p+r-\eta-2}>1$ 
by the condition $r<2$. 
Proceeding similarly as the above estimate we obtain that 
for all $\ep_{16}>0$ there exists $\theta_{16}>1$ such that
%
\begin{align}
      I_{16}
      &\le B_{12} \cdot \ep_{16}^{\theta_{16}}\int_\Omega
      \frac{(u+1)^{p-\eta}}{w^{2k_2+\sigma_2}}|\nabla w|^2
      +B_{12} \cdot \ep_{16}^{-\frac{\theta_{16}}{\theta_{16}-1}}
      \int_\Omega |\nabla w|^2.
\end{align}
%

{\bf Step 4.} We estimate the four terms which do not contain 
$|\nabla u|, |\nabla v|, |\nabla w|$ (i.e., $I_{22}$--$I_{25}$) by 
$\int_\Omega (u+1)^p$. 
Indeed, it suffices to note that 
%
\begin{align}\label{estI22}
      \int_\Omega (u+1)^{p-\eta}
      &\le \frac{p-\eta}{p}\int_\Omega (u+1)^p+\frac{\eta}{p}, 
\end{align}
%
which leads to the required estimates. 

\vspace{2mm}
Thus, in view of Steps 1--4, the estimates for 
the twenty-five terms $I_1$--$I_{25}$ in \eqref{DI2-12} are complete. 
We finally derive 
\eqref{DI2}. 
Combining \eqref{DI2-12} and \eqref{estI1}--\eqref{estI22}, 
we have
%
\begin{align*}
    &\quad\,\ep_{03}\frac{d}{dt}\int_\Omega 
      \frac{(u+1)^{p-\eta}}
             {v^{2k_1+\sigma_3-2}w^{2k_2+\sigma_4-2}}
    +\frac{d}{dt}\int_\Omega 
      \frac{(u+1)^{p-\eta}}{v^{2k_1+\sigma_1-2}}
    +\frac{d}{dt}\int_\Omega 
      \frac{(u+1)^{p-\eta}}{w^{2k_2+\sigma_2-2}}\notag\\
    &\quad
    +B_{19}\int_\Omega 
      \frac{(u+1)^{p-\eta}}{v^{2k_1+\sigma_1}}
      |\nabla v|^2
    +B_{20}\int_\Omega 
      \frac{(u+1)^{p-\eta}}{w^{2k_2+\sigma_2}}
      |\nabla w|^2
    \notag\\
    &\quad
    +\ep_{03}A_{11}\int_\Omega 
      \frac{(u+1)^{p-\eta}}{v^{2k_1+\sigma_3}w^{2k_2+\sigma_4-2}}|\nabla v|^2
    +\ep_{03}A_{17}\int_\Omega 
      \frac{(u+1)^{p-\eta}}{v^{2k_1+\sigma_3-2}w^{2k_2+\sigma_4}}
      |\nabla w|^2
    \notag\\
    &\le c_1\int_\Omega 
      \frac{(u+1)^{p-\eta}}{v^{2k_1+\sigma_1}}
      |\nabla v|^2
      +c_2\int_\Omega 
      \frac{(u+1)^{p-\eta}}{w^{2k_2+\sigma_2}}
      |\nabla w|^2
      \notag\\
      &\quad+c_3\int_\Omega 
      \frac{(u+1)^{p-\eta}}{v^{2k_1+\sigma_3}w^{2k_2+\sigma_4-2}}|\nabla v|^2
      +c_4\int_\Omega 
      \frac{(u+1)^{p-\eta}}{v^{2k_1+\sigma_3-2}w^{2k_2+\sigma_4}}
      |\nabla w|^2
      \notag\\
      &\quad 
      +c_5\int_\Omega (u+1)^{p-2}|\nabla u|^2
      +c_6\int_\Omega |\nabla v|^2
      +c_7\int_\Omega |\nabla w|^2
      +c_8\int_\Omega (u+1)^p+c_9,
\end{align*}
%
where
%
\begin{align*}
   c_1&:=\frac{1}{2}(\ep_{03}A_{12}'
            +B_3'\ep_2^{\theta_2-1}
            +B_{10}'\ep_{13}^{\theta_{13}-1}
            +B_{18}'\ep_{18}^{\theta_{18}-1}),
   \notag\\
   c_2&:=\frac{1}{2}(\ep_{03}A_{12}'
            +B_4'\ep_4^{\theta_4-1}
            +B_{11}'\ep_{14}^{\theta_{14}-1}
            +B_{12}\ep_{16}^{\theta_{16}}
            +B_{14}'\ep_{20}^{\theta_{20}-1}),
   \notag\\
   c_3&:=\frac{1}{2}(\ep_{03}A_{4}'\ep_1^{\theta_1-1}
            +\ep_{03}A_{10}'\ep_{17}^{\theta_{17}-1}
            +B_7'\ep_7^{\theta_7-1}
            +B_{11}\ep_{14}),
   \notag\\
   c_4&:=\frac{1}{2}(\ep_{03}A_{7}'\ep_3^{\theta_3-1}
            +\ep_{03}A_{9}'\ep_{15}^{\theta_{15}}
            +\ep_{03}A_{15}'\ep_{19}^{\theta_{19}-1}
            +B_9'\ep_{10}^{\theta_{10}-1}
            +B_{10}\ep_{13}
            ),
   \notag\\
   c_5&:=\frac{1}{2}(
            \ep_{03}A_4\ep_1
            +\ep_{03}A_7\ep_3
            +\ep_{03}A_{10}\ep_{17}
            +\ep_{03}A_{15}\ep_{19}
   \notag\\
   &\qquad\qquad
            +B_3\ep_2
            +B_4\ep_4
            +B_7\ep_7
            +B_9\ep_{10}
            +B_{13}\ep_{18}
            +B_{14}\ep_{20}
            ),
   \notag\\
   c_6&:=\frac{1}{2}(\ep_{03}A_4''\ep_1^{-\frac{\theta_1}{\theta_1-1}-1}
            +\ep_{03}A_{10}''\ep_{17}^{-\frac{\theta_{17}}{\theta_{17}-1}-1}
   \notag\\
   &\qquad\qquad
            +B_3''\ep_2^{-\frac{\theta_2}{\theta_2-1}-1}
            +B_7''\ep_7^{-\frac{\theta_7}{\theta_7-1}-1}
            +B_{10}''\ep_{13}^{-\frac{\theta_{13}}{\theta_{13}-1}-1}
            +B_{13}''\ep_{18}^{-\frac{\theta_{18}}{\theta_{18}-1}-1}
            ),
   \notag\\
   c_7&:=\frac{1}{2}(\ep_{03}A_7''\ep_3^{-\frac{\theta_3}{\theta_3-1}-1}
            +\ep_{03}A_{9}''\ep_{15}^{-\frac{\theta_{15}}{\theta_{15}-1}-1}
            +\ep_{03}A_{15}''\ep_{19}^{-\frac{\theta_{19}}{\theta_{19}-1}-1}
   \notag\\
   &\qquad\qquad
            +B_4''\ep_4^{-\frac{\theta_4}{\theta_4-1}-1}
            +B_9''\ep_{10}^{-\frac{\theta_{10}}{\theta_{10}-1}-1}
            +B_{11}''\ep_{14}^{-\frac{\theta_{14}}{\theta_{14}-1}-1}
            +B_{12}''\ep_{16}^{-\frac{\theta_{16}}{\theta_{16}-1}-1}
            +B_{14}''\ep_{20}^{-\frac{\theta_{20}}{\theta_{20}-1}-1}
            ),
   \notag\\
   c_8&:=\frac{p-\eta}{p}(
            \ep_{03}A_{14}'+\ep_{03}A_{19}'+B_{16}'+B_{18}'
            ),
   \notag\\
   c_9&:=\frac{\eta}{p}(
            \ep_{03}A_{14}'+\ep_{03}A_{19}'+B_{16}'+B_{18}'
            ).
\end{align*}
%
Consequently, by choosing $\ep_{03}, \ep_i>0$ 
($i \in \{1, 2, 3, 4, 7, 10, 13, 14, 15, 16, 17, 18, 19, 20\}$) 
satisfying $B_{19}-c_1>0$, $B_{20}-c_2>0$, $\ep_{03}A_{11}=c_3$, 
$\ep_{03}A_{17}=c_4$ and $c_5=\frac{a_0(p-m)(p-m+1)}{4}$, 
we derive the differential inequality \eqref{DI2}.
\end{proof}

\begin{prth1.1}
Setting $\ep_{01}:=C_3$ and $\ep_{02}:=C_4$ in Lemma~\ref{difineq1}, 
where $C_3, C_4>0$ are constants appearing in Lemma~\ref{difineq2}, 
we know that
\begin{align}\label{pfDI}
&\quad\,\frac{d}{dt}\left(
\int_\Omega (u+1)^{p-m+1}
+\int_\Omega \frac{(u+1)^{p-\eta}}{v^{2k_1+\sigma_1-2}}
+\int_\Omega \frac{(u+1)^{p-\eta}}{w^{2k_2+\sigma_2-2}}
+\ep_{03}\int_\Omega \frac{(u+1)^{p-\eta}}{v^{2k_1+\sigma_3-2}w^{2k_2+\sigma_4-2}}
\right)\notag\\ 
&\qquad\, \quad
+\frac{a_0(p-m)(p-m+1)}{4}\int_\Omega (u+1)^{p-2}|\nabla u|^2\notag\\[2mm]
&\le 
c_1\int_\Omega |\nabla v|^2+c_2\int_\Omega |\nabla w|^2
+c_3\int_\Omega (u+1)^p+c_4
\end{align}
for all $t \in (0, \tmax)$ with some $c_1, c_2, c_3, c_4>0$, $\ep_{03}>0$. 
Also, multiplying the second and third equations in \eqref{P} by 
$v$ and $w$, respectively, and integrating them over $\Omega$, 
we have
\begin{align}
&\frac{d}{dt}\int_\Omega v^2+2d_1\int_\Omega |\nabla v|^2
\le \frac{\alpha^2}{\beta}\int_\Omega u^2-\beta\int_\Omega v^2,
\label{vDI}\\ 
&\frac{d}{dt}\int_\Omega w^2+2d_2\int_\Omega |\nabla w|^2
\le \frac{\gamma^2}{\delta}\int_\Omega u^2-\delta\int_\Omega w^2.
\label{wDI}
\end{align}
Multiplying \eqref{vDI} and \eqref{wDI} by $\frac{c_1}{2d_1}$ and 
$\frac{c_2}{2d_2}$, respectively, and adding them to \eqref{pfDI}, we obtain
\begin{align}\label{pfDI2}
&\quad\,\frac{d}{dt}\left(
\int_\Omega (u+1)^{p-m+1}
+\int_\Omega \frac{(u+1)^{p-\eta}}{v^{2k_1+\sigma_1-2}}
+\int_\Omega \frac{(u+1)^{p-\eta}}{w^{2k_2+\sigma_2-2}}
+\ep_{03}\int_\Omega \frac{(u+1)^{p-\eta}}{v^{2k_1+\sigma_3-2}w^{2k_2+\sigma_4-2}}\right.\notag\\
&\qquad\,\quad \left.+\frac{c_1}{2d_1}\int_\Omega v^2+\frac{c_2}{2d_2}\int_\Omega w^2
\right)+\frac{a_0(p-m)(p-m+1)}{4}\int_\Omega (u+1)^{p-2}|\nabla u|^2
\notag\\[2mm]
&\le  c_3\int_\Omega (u+1)^p+\Big(\frac{c_1\alpha^2}{2d_1\beta}+\frac{c_2\gamma^2}{2d_2\delta}\Big)\int_\Omega u^2
-\frac{c_1\beta}{2d_1}\int_\Omega v^2
-\frac{c_2\delta}{2d_2}\int_\Omega w^2+c_4
\end{align}
for all $t \in (0, \tmax)$. 
By virtue of the Gagliardo--Nirenberg inequality, 
we see that
\begin{align*}
c_3\int_\Omega (u+1)^p 
&=c_3\|(u+1)^{\frac p2}\|_{L^2(\Omega)}^2\notag\\
&\le c_5\Big(\|\nabla(u+1)^{\frac{p}{2}}\|_{L^2(\Omega)}^{\theta_1}
                    \|(u+1)^{\frac p2}\|_{L^{\frac 2p}(\Omega)}^{1-\theta_1}
                +\|(u+1)^{\frac p2}\|_{L^{\frac 2p}(\Omega)}\Big)^2
\end{align*} 
with some $c_5>0$ and $\theta_1:=\frac{pn-n}{pn+2-n} \in (0,1)$. 
Here, noting from the first equation in \eqref{P} that 
the mass conservation $\int_\Omega u(\cdot, t)=\int_\Omega u_0$ holds 
for all $t \in (0, \tmax)$ and using Young's inequality, 
we derive
\begin{align}\label{GN2}
c_3\int_\Omega (u+1)^p 
&\le c_6\Big(\|\nabla(u+1)^{\frac{p}{2}}\|_{L^2(\Omega)}^{\theta_1}+1\Big)^2\notag\\
&\le \frac{a_0(p-m)(p-m+1)}{16}\int_\Omega (u+1)^{p-2}|\nabla u|^2+c_7
\end{align} 
with some $c_6, c_7>0$. 
Also, recalling the lower estimates \eqref{infv} and \eqref{infw}, 
we infer from \eqref{GN2} that 
\begin{align}\label{pfDI5}
&\int_\Omega \frac{(u+1)^{p-\eta}}{v^{2k_1+\sigma_1-2}}
+\int_\Omega \frac{(u+1)^{p-\eta}}{w^{2k_2+\sigma_2-2}}
+\ep_{03}\int_\Omega \frac{(u+1)^{p-\eta}}{v^{2k_1+\sigma_3-2}w^{2k_2+\sigma_4-2}}\notag\\
&\qquad\,
\le \Big(\frac{1}{\mu_1^{2k_1+\sigma_1-2}}
+\frac{1}{\mu_2^{2k_2+\sigma_2-2}}
+\frac{\ep_{03}}{\mu_1^{2k_1+\sigma_1-2}\mu_2^{2k_2+\sigma_2-2}}\Big)\int_\Omega (u+1)^{p-\eta}\notag\\
&\qquad\,\le c_3\int_\Omega (u+1)^p+c_8\notag\\
&\qquad\,\le \frac{a_0(p-m)(p-m+1)}{16}\int_\Omega (u+1)^{p-2}|\nabla u|^2+c_9
\end{align}
with some $c_8, c_9>0$. 
Moreover, we derive from the relation $u^2 \le (u+1)^p$ and \eqref{GN2} that 
\begin{align}\label{pfDI8}
\Big(\frac{c_1\alpha^2}{2d_1\beta}+\frac{c_2\gamma^2}{2d_2\delta}\Big)\int_\Omega u^2
\le\frac{a_0(p-m)(p-m+1)}{16}\int_\Omega (u+1)^{p-2}|\nabla u|^2+c_{10}
\end{align}
with some $c_{10}>0$. 
Collecting \eqref{GN2}--\eqref{pfDI8} in \eqref{pfDI2}, we establish
\begin{align}\label{pfDI9}
&\quad\,\frac{d}{dt}\left(
\int_\Omega (u+1)^{p-m+1}
+\int_\Omega \frac{(u+1)^{p-\eta}}{v^{2k_1+\sigma_1-2}}
+\int_\Omega \frac{(u+1)^{p-\eta}}{w^{2k_2+\sigma_2-2}}
+\ep_{03}\int_\Omega \frac{(u+1)^{p-\eta}}{v^{2k_1+\sigma_3-2}w^{2k_2+\sigma_4-2}}\right.\notag\\
&\qquad\, \left.+\frac{c_1}{2d_1}\int_\Omega v^2+\frac{c_2}{2d_2}\int_\Omega w^2
\right)\notag\\
&\qquad\,+\frac{a_0(p-m)(p-m+1)}{16}\int_\Omega (u+1)^{p-2}|\nabla u|^2
+\int_\Omega \frac{(u+1)^{p-\eta}}{v^{2k_1+\sigma_1-2}}
+\int_\Omega \frac{(u+1)^{p-\eta}}{w^{2k_2+\sigma_2-2}}\notag\\
&\qquad\,+\ep_{03}\int_\Omega \frac{(u+1)^{p-\eta}}{v^{2k_1+\sigma_3-2}w^{2k_2+\sigma_4-2}}+\frac{c_1\beta}{2d_1}\int_\Omega v^2
+\frac{c_2\delta}{2d_2}\int_\Omega w^2\notag\\
&\le  c_{11}
\end{align}
for all $t \in (0, \tmax)$ with some $c_{11}>0$. 
Here we estimate the term $\int_\Omega (u+1)^{p-2}|\nabla u|^2$. 
Again by the Gagliardo--Nirenberg inequality, we have
\begin{align*}
\int_\Omega (u+1)^{p-m+1}
&=\|(u+1)^{\frac p2}\|_{L^{\frac{2(p-m+1)}{p}(\Omega)}}^{\frac{2(p-m+1)}{p}}\notag\\
&\le c_{12}\Big(\|\nabla(u+1)^{\frac{p}{2}}\|_{L^2(\Omega)}^{\theta_2}
                \|(u+1)^{\frac p2}\|_{L^{\frac 2p}(\Omega)}^{1-\theta_2}
                +\|(u+1)^{\frac p2}\|_{L^{\frac 2p}(\Omega)}\Big)^{\frac{2(p-m+1)}{p}}
\end{align*}
with some $c_{12}>0$ and $\theta_2:=\frac{(p-m)pn}{(p-m+1)(pn+2-n)} 
\in (0,1)$ for sufficiently large $p$ fulfilling \eqref{condip2}. 
This together with the mass conservation yields
\begin{align*}
\int_\Omega (u+1)^{p-m+1}
&\le c_{13}\Big(\|\nabla(u+1)^{\frac{p}{2}}\|_{L^2(\Omega)}^{\theta_2}
                +1\Big)^{\frac{2(p-m+1)}{p}}
\end{align*}
with some $c_{13}>0$ and hence
\begin{align}\label{GN3}
c_{14}\Big(\int_\Omega (u+1)^{p-m+1}\Big)^{\theta_3} 
\le \frac{a_0(p-m)(p-m+1)}{16}\int_\Omega (u+1)^{p-2}|\nabla u|^2+c_{15}
\end{align}
with some $c_{14}, c_{15}>0$ and $\theta_3:=\frac{pn+2-n}{(p-m)n}>0$. 
Combining \eqref{GN3} with \eqref{pfDI9}, we obtain
\begin{align}\label{pfDI10}
&\quad\,\frac{d}{dt}\left(
\int_\Omega (u+1)^{p-m+1}
+\int_\Omega \frac{(u+1)^{p-\eta}}{v^{2k_1+\sigma_1-2}}
+\int_\Omega \frac{(u+1)^{p-\eta}}{w^{2k_2+\sigma_2-2}}
+\ep_{03}\int_\Omega \frac{(u+1)^{p-\eta}}{v^{2k_1+\sigma_3-2}w^{2k_2+\sigma_4-2}}\right.\notag\\
&\qquad\, \left.+\frac{c_1}{2d_1}\int_\Omega v^2+\frac{c_2}{2d_2}\int_\Omega w^2
\right)\notag\\
&\qquad\,+c_{14}\Big(\int_\Omega (u+1)^{p-m+1}\Big)^{\theta_3} 
+\int_\Omega \frac{(u+1)^{p-\eta}}{v^{2k_1+\sigma_1-2}}
+\int_\Omega \frac{(u+1)^{p-\eta}}{w^{2k_2+\sigma_2-2}}\notag\\
&\qquad\,+\ep_{03}\int_\Omega \frac{(u+1)^{p-\eta}}{v^{2k_1+\sigma_3-2}w^{2k_2+\sigma_4-2}}+\frac{c_1\beta}{2d_1}\int_\Omega v^2
+\frac{c_2\delta}{2d_2}\int_\Omega w^2\notag\\
&\le  c_{16}
\end{align}
for all $t \in (0, \tmax)$ with some $c_{16}>0$. 
Putting 
\begin{align*}
y(t)&:=\int_\Omega (u(\cdot, t)+1)^{p-m+1}
+\int_\Omega \frac{(u(\cdot, t)+1)^{p-\eta}}{v^{2k_1+\sigma_1-2}(\cdot, t)}
+\int_\Omega \frac{(u(\cdot, t)+1)^{p-\eta}}{w^{2k_2+\sigma_2-2}(\cdot, t)}
\notag\\
&\quad\,
+\ep_{03}\int_\Omega 
\frac{(u(\cdot, t)+1)^{p-\eta}}
       {v^{2k_1+\sigma_3-2}(\cdot, t)w^{2k_2+\sigma_4-2}(\cdot, t)}
+\frac{c_1}{2d_1}\int_\Omega v^2(\cdot, t)+\frac{c_2}{2d_2}\int_\Omega w^2(\cdot, t)\quad {\rm for}\ t>0,
\end{align*}
we see from \eqref{pfDI10} that 
\begin{align*}
y'(t)+c_{17}y^\kappa(t) \le c_{18}
\end{align*}
for all $t \in (0, \tmax)$ with some $c_{17}, c_{18}>0$ and 
$\kappa:=\min\{\theta, 1\}$. 
Thus we have
\begin{align*}
\sup_{t \in [0,\tmax)}\int_\Omega (u(\cdot, t)+1)^{p-m+1}<\infty
\end{align*}
for sufficiently large $p$ satisfying \eqref{condip2}. 
This yields $\sup_{t \in [0, \tmax)}\|u(\cdot, t)\|_{L^\infty(\Omega)}<\infty$ 
(see \cite[Lemma~A.1]{TW-2012})
which leads to $\tmax=\infty$. 
Therefore we arrive at the conclusion. \qed
\end{prth1.1}


\end{document}